\documentclass[10pt]{amsart}
\allowdisplaybreaks

\usepackage{amsfonts}
\usepackage{amsmath}
\usepackage{amssymb}
\usepackage{amsthm}
\usepackage{graphicx} \usepackage{enumerate} \usepackage{multicol}
\usepackage{mathrsfs} \usepackage[all,cmtip]{xy}
\usepackage{enumerate}
\usepackage{xcolor}

\newcounter{dummy} \numberwithin{dummy}{section}
\newtheorem{theorem}[dummy]{Theorem}

\newtheorem{lemma}[dummy]{Lemma}
\newtheorem{definition}[dummy]{Definition}
\newtheorem{proposition}[dummy]{Proposition}
\theoremstyle{remark}
\newtheorem{remark}[dummy]{Remark}
\newtheorem{example}[dummy]{Example}

\newcommand\transpose{%
  {\mathchoice
    {\raisebox{.45ex}{$\displaystyle{\intercal}$}}
    {\raisebox{.45ex}{$\textstyle{\intercal}$}}
    {\raisebox{.30ex}{$\scriptstyle{\intercal}$}}  
    {\raisebox{.23ex}{$\scriptscriptstyle{\intercal}$}}}
  }

\newcommand{\calE}{\mathcal{E}}

\newcommand{\calO}{\mathcal{O}}

\newcommand{\sfF}{\mathsf{F}}
\newcommand{\sff}{\mathsf{f}}
\newcommand{\sfc}{\mathsf{c}}
\newcommand{\sfC}{\mathsf{C}}
\newcommand{\sfG}{\mathsf{G}}
\newcommand{\sfSigma}{\mathsf{\Sigma}}

\DeclareMathOperator{\Ann}{Ann}
\DeclareMathOperator{\Sym}{Sym}

\DeclareMathOperator{\id}{id}

\DeclareMathOperator{\pr}{pr}
\DeclareMathOperator{\rank}{rank}
\DeclareMathOperator{\spn}{span}

\DeclareMathOperator{\Ad}{Ad}

\DeclareMathOperator{\GL}{GL}
\DeclareMathOperator{\gl}{\mathfrak{gl}}
\DeclareMathOperator{\SO}{SO}
\DeclareMathOperator{\Ort}{O}
\DeclareMathOperator{\ort}{\mathfrak{o}}
\DeclareMathOperator{\Aut}{Aut}

\DeclareMathOperator{\Isom}{Isom}
\DeclareMathOperator{\Hol}{Hol}
\DeclareMathOperator{\Nil}{Nil}
\DeclareMathOperator{\Frame}{Fr}
\DeclareMathOperator{\nil}{\mathfrak{nil}}
\DeclareMathOperator{\ptr}{/\! \! /}
\newcommand{\bnabla}{\boldsymbol \nabla}

\newcommand\blank{{\kern.8pt\cdot\kern.8pt}}

\numberwithin{equation}{section}

\title[Model spaces in sub-Riemannian geometry]{Model spaces in sub-Riemannian geometry}
\author[E.~Grong]{Erlend Grong}

\address{Universit\'e Paris Sud, Laboratoire des Signaux et Syst\`emes (L2S) Sup\'elec, CNRS, Universit\'e Paris-Saclay, 3 rue Joliot-Curie, 91192
Gif-sur-Yvette, France and University of Bergen, Department of Mathematics, P. O. Box 7803,
5020 Bergen, Norway.}
\email{erlend.grong@gmail.com}



\begin{document}
\begin{abstract}
We consider sub-Riemannian spaces admitting an isometry group that is maximal in the sense that any linear isometry between the horizontal tangent spaces is realized by a global isometry. We will show that these spaces have a canonical choice of partial connection on their horizontal bundle, which is determined by isometries and generalizes the Levi-Civita connection for the special case of Riemannian model spaces. The number of invariants needed to describe model spaces with the same tangent cone is in general greater than one, and these invariants are not necessarily related to the holonomy of the canonical connections.
 \end{abstract}

\maketitle
\section{Introduction} \label{sec;Introduction}
One need only look to the classical Gauss map \cite{Gau65} to see that the development of Riemannian geometry is informed by the ideas of model spaces. The euclidean space, the hyperbolic spaces and the spheres are also the reference spaces for comparison theorems in Riemannian geometry, such as the Laplacian comparison theorem \cite[Theorem~3.4.2]{Hsu02} and volume comparison theorem \cite[Chapter 9]{Pet06}. For sub-Riemannian geometry, the lack of a generalization of the Levi-Civita connection has complicated the understanding of geometric invariants. One has had the idea of `flat space' since the description by Mitchell on sub-Riemannian tangent cones \cite{Mit85}, later improved in \cite{Bel96}, but it is less clear what the sub-Riemannian analogues of spheres and hyperbolic spaces should be. As there has been recent investigations into generalizations of curvature to the sub-Riemannian setting, it seems important to establish model spaces as a reference. See \cite{BaGa17,BaBo12,BBG14,BKW14,GrTh14a,GrTh14b} for approaches to curvature by studying properties of the heat flow and ~\cite{ZeLi07,LiZe11,ABR13,BaRi16} for an approach using the sub-Riemannian geodesic flow. We mention also one of the earliest considerations of curvature in the case of three-dimensional contact structures in Keener Hughen's thesis \cite{Hug95} using Cartan's method of equivalence.

We want to consider sub-Riemannian model spaces from the point of view of spaces with maximal isometry groups. Intuitively, these are sub-Riemannian manifolds that `look the same' not only at every point, but for every orientation of the horizontal bundle. In particular, we want to investigate sub-Riemannian geometry by determining invariants of model spaces with the same tangent cone. For a metric space $(M, d_M)$ and a point~$x_0 \in M$, we say that the pointed space $(N, d_N, y_0)$ is the tangent cone of $M$ at~$x_0$ if any ball of radius $r$ centered at $x_0$ in the space $(M, \lambda d_M)$ converge to a ball of radius $r$ centered at $y_0$ in the Gromov-Hausdorff distance as $\lambda \to \infty$. Any $n$-dimensional Riemannian manifold has the $n$-dimensional euclidean space as its tangent cone at any fixed point. Hence, any Riemannian model space is uniquely determined by its tangent cone and one parameter, its sectional curvature. We will show that the number of invariants determining sub-Riemannian model spaces with the same tangent cone is not necessarily one. Furthermore, these parameters may in general not relate to whether or not the model space has trivial holonomy.

To be more specific, we list our main results for a sub-Riemannian model space $(M, D, g)$, where $D$ is the horizontal bundle of rank $n$ and $g$ is the sub-Riemannian metric defined on~$D$.
\begin{enumerate}[\rm (i)]
\item On any model space, there exists a unique partial connection on the horizontal bundle $D$ which is invariant under the isometry group,which equals the Levi-Civita connection for the case $D = TM$. We call this the canonical partial connection, see Section~\ref{sec:Canonical}.
\item The horizontal holonomy group of the canonical partial connection is either trivial or isomorphic to $\SO(n)$. If the holonomy is trivial, then $M$ is a Lie group and all isometries are compositions of left translations and Lie group automorphisms. If the holonomy is isomorphic to $\SO(n)$, then we obtain a new sub-Riemannian model space $\Frame(M)$ with trivial holonomy by considering a lifted structure to the orthonormal frame bundle of $D$. See Sections~\ref{sec:Holonomy}, \ref{sec:Lie} and \ref{sec:Frame}.
\item The tangent cone of $M$ at any point is a Carnot group that is also a model space. See Section~\ref{sec:Tangent}.
\item If $M$ is a model space of even step whose tangent cone equals the free nilpontent Lie group, then $M$ is a Lie group. See Theorem~\ref{th:Free}.
\item Any sub-Riemannian model space of step~$2$ is either the free nilpotent group or locally isometric to the frame bundle of curved Riemannian model space. In particular, in step~$2$, any model space is uniquely determined by its tangent cone and one parameter; however, all have trivial holonomy. See Theorem~\ref{th:Step2}.
\item We give an example in step~3, where all the model spaces of a given tangent cone is determined by two parameters. Only one of these parameters determines the horizontal holonomy of the canonical partial connection. See Theorem~\ref{th:Rol3}.
\item There is a binary operation $(M^{(1)}, M^{(2)}) \mapsto M^{(1)} \boxplus M^{(2)}$ of sub-Riemannian model spaces that is commutative, associative and distribute. See Section~\ref{sec:RolSum}.
\end{enumerate}

The structure of the paper is as follows. In Section~\ref{sec:Prelim}, we include some preliminaries and define sub-Riemannian model spaces. We prove that these spaces have a canonical partial connection in Section~\ref{sec:Connection} and discuss holonomy and curvature of such connections. Next, we consider Carnot groups who are also model spaces in Section~\ref{sec:Carnot}. We choose two classes of Carnot groups and study model spaces with these groups as their tangent cone more closely in Section~\ref{sec:Free} and~\ref{sec:Rol}. We include some basic results on representations of $\Ort(n)$ in Appendix~\ref{sec:On} needed for the results in Section~\ref{sec:Free} and~\ref{sec:Rol}.

\subsection*{Acknowledgments} 
We thank professors Enrico Le Donne and Jean-Marc Schlenker for helpful discussions and comments. We also thank the referees of this paper for their comments on how to improve the paper. This work has been supported by the Fonds National de la Recherche Luxembourg (project Open O14/7628746 GEOMREV) and by the Research Council of Norway (project number 249980/F20).

\section{Preliminaries} \label{sec:Prelim}
\subsection{Notation}
If $\mathfrak{a}$ and $\mathfrak{b}$ are two inner product spaces, we write $\Ort(\mathfrak{a}, \mathfrak{b})$ for the space of all linear isometries from $\mathfrak{a}$ to $\mathfrak{b}$, and write the Lie group $\Ort(\mathfrak{a},\mathfrak{a})$ as just $\Ort(\mathfrak{a})$. When $\mathfrak{a}$ and $\mathfrak{b}$ are oriented, we define $\SO(\mathfrak{a}, \mathfrak{b})$ and $\SO(\mathfrak{a})$ analogously. We write $\ort(n)$ for the Lie algebra of $\Ort(n)$, consisting of anti-symmetric matrices. In what follows, it will often be practical to identify $\wedge^2 \mathbb{R}^n$ and~$\ort(n)$ as vector spaces through the map 
\begin{equation} \label{wedge} x \wedge y \mapsto y x^\transpose - x y^\transpose , \qquad x,y \in \mathbb{R}^n.
\end{equation}
Relative to this identification, note that  for any $x,y \in \mathbb{R}^n$, $A \in \ort(n)$ and $a \in \Ort(n)$, we have
$$[A, x \wedge y] = Ax \wedge y + x \wedge Ay \quad \text{and} \quad \Ad(a) x \wedge y = ax \wedge a y.$$

If $\mathfrak{a}$ and $\mathfrak{b}$ are two vector spaces, possibly with Lie algebra structures, then the notation $\mathfrak{g} = \mathfrak{a} \oplus \mathfrak{b}$ will only mean that~$\mathfrak{g}$ equals $\mathfrak{a} \oplus \mathfrak{b}$ as a vector space, stating nothing about a possible Lie algebra structure.

If $\pi: E \to M$ is a vector bundle, we will write $\Gamma(E)$ for the space of its smooth sections. If $g \in \Gamma(\Sym^2 E^*)$ is a metric tensor on $E$, we write $g(e_1, e_2) = \langle e_1, e_2 \rangle_g$ and $| e_1 |_g = \langle e_1, e_1\rangle^{1/2}_g$ for $e_1 ,e_2 \in E$.

\subsection{Sub-Riemannian manifolds} \label{sec:SRdef}
Let $M$ be a connected manifold and let $D$ be a subbundle of $TM$ equipped with a positive definite metric tensor~$g$. The pair $(D,g)$ is then called \emph{a sub-Riemannian structure} and the triple $(M, D, g)$ is called \emph{a sub-Riemannian manifold}. An absolutely continuous curve $\gamma$ in $M$ is called \emph{horizontal} if $\dot \gamma(t) \in D_{\gamma(t)}$ for almost every $t$. 
Associated with the sub-Riemannian structure $(D, g)$, we have \emph{the Carnot-Carath\'eodory metric}
$$d_M(x, y) = \inf \left\{ \int_0^1 | \dot \gamma(t) |_g \, dt \, : \, \begin{array}{c} \text{$\gamma$ horizontal} \\ \text{$\gamma(0) = x$, $\gamma(1) = y$.}\end{array} \right\},$$
with $x,y \in M$. If any pair of points can be connected by a horizontal curve, then $(M, d_M)$ is a well defined metric space. A sufficient condition for connectivity by horizontal curves is that $D$ is \emph{bracket-generating}, i.e. that its sections and their iterated brackets span the entire tangent bundle. Not only will $d_M$ be well-defined in this case, but its topology coincides with the manifold topology. For more on sub-Riemannian manifolds, see \cite{Mon02}.

We next turn to the definition of isometries, which will be central to our discussion of model spaces.
\begin{definition}
For $j =1,2$, let $(M^{(j)}, D^{(j)}, g^{(j)})$ be sub-Riemannian manifolds with $D^{(j)}$ bracket-generating. Let $d_{(j)}$ be the Carnot-Carath\'eodory metric of $g^{(j)}$. A homeomorphism $\varphi:M^{(1)} \to M^{(2)}$ is called a sub-Riemannian isometry if
$$d_{(2)}(\varphi(x), \varphi(y)) = d_{(1)}(x,y),$$
for any $x,y \in M^{(1)}$.
\end{definition}
We know that isometries are always smooth maps for the case when $D$ is \emph{equiregular}. Let $D$ be a general subbundle of $TM$. Define $\underline{D}^1 = \Gamma(D)$ and for $j \geq 1$,
\begin{equation} \label{Dunderline} \underline{D}^{j+1} : = \spn \left\{Y, [X,Y] \, : \, X \in \Gamma(D), Y \in \underline{D}^j \right\}.\end{equation}
For each $j \geq 1$ and $x \in M$, write $n_j(x) = \rank \{ Y(x) \, : \,  Y \in \underline{D}^j \}$. The sequence $\underline{n}(x) = (n_1(x), n_2(x), \dots)$ is called \emph{the growth vector of $D$ at $x$}.  The subbundle $D$ is called equiregular if the functions $n_j$ are constant. We will refer to the minimal integer $r$ such that $\underline{D}^r = \underline{D}^{r+1}$ as the \emph{step} of $D$. If $D$ is equiregular, there exists a flag of subbundles
\begin{equation} \label{flag}  D^0 =0 \subseteq D^1 = D  \subseteq D^2 \subseteq \cdots \subseteq D^r,\end{equation}
such that $\underline{D}^j = \Gamma(D^j)$. The following result is found in \cite{CaLD13}.

\begin{theorem} \label{th:Isom}
Let $(M, D, g)$ be a sub-Riemannian manifold with $D$ bracket-generating and equiregular and let $\varphi: M \to M$ be an isometry. Then $\varphi$ is a smooth map saitisfying $\varphi_* (D)\subseteq D$ and $\langle \varphi_* v, \varphi_* w \rangle_g = \langle v,w \rangle_g$. Furthermore, if $\tilde \varphi$ is another isometry such that for some $x \in M$,
$$\tilde \varphi(x) = \varphi(x) \quad \text{and} \quad \tilde \varphi_* |D_x = \varphi_*|D_x,$$
then $\tilde \varphi = \varphi$. The group of all isometries admits the structure of a finite dimensional Lie group.
\end{theorem}
In what follows, we will often write the sub-Riemannian manifold $(M, D, g)$ simply as $M$ if $(D,g)$ is understood from the context. This includes the case of Riemannian manifolds $D = TM$. Also, for any $\lambda > 0$, the space $(M, D, \lambda^2 g)$ is denoted by $\lambda M$. Clearly, if $d_M$ is the metric of $M$, the space $\lambda M$ has metric $\lambda d_M$.

\begin{remark} \label{re:Dj}  Let $\varphi:M \to M$ be a diffeomorphism such that $\varphi_*(D) \subseteq D$. Define $\Ad(\varphi) X$ by $\Ad(\varphi) X(x) = \varphi_* X(\varphi^{-1}(x))$ for any vector field $X \in \Gamma(TM)$ and point $x \in M$. Then $[\Ad(\varphi) X, \Ad(\varphi) Y] = \Ad(\varphi) [X,Y]$, so $\Ad(\varphi) \underline{D}^j \subseteq \underline{D}^j$ for any $j \geq 1$. As a consequence, if $D$ is equiregular and $\underline{D}^j = \Gamma(D^j)$, we have $\varphi_* D^j \subseteq D^j$ as well.
\end{remark}

\subsection{Riemannian model spaces} \label{sec:Riemannian}
A complete, connected, simply connected Riemannian manifold $(M, g)$ is called \emph{a model space} if it has constant sectional curvature. We review some basic facts about such spaces and refer to \cite{KoNo63} for details. For a given dimension $n$ and sectional curvature~$\rho$, there is a unique model space up to isometry, which we will denote by $\sfSigma_n(\rho)$. The cases $\rho = 0$, $\rho > 0$ and $\rho < 0$ correspond respectively to the cases of euclidean space, spheres and hyperbolic space.

These model spaces can be seen as symmetric spaces. For any integer $n \geq 2$ and constant $\rho \in \mathbb{R}$, consider the Lie algebra $\mathfrak{g} = \mathsf{g}_n(\rho)$ of all matrices
$$\left(\begin{array}{cc} A & x \\ -\rho x^\transpose & 0 \end{array} \right), \qquad \text{$A \in \ort(n)$, $x \in \mathbb{R}^n$.}$$
Write $\mathfrak{g} = \mathfrak{p} \oplus \mathfrak{k}$ corresponding to the subspaces $A = 0$ and $x = 0$ respectively. Give $\mathfrak{p}$ the standard euclidean metric in these coordinates. Let $G = \sfG_n(\rho)$ be the corresponding simply connected Lie group and let $K$ be the subgroup corresponding to the subalgebra $\mathfrak{k}$. Define a sub-Riemannian structure $( D, g)$ on $G$ by left translation of $\mathfrak{p}$ and its inner product. Since both $\mathfrak{p}$ and its inner product are $K$-invariant, this sub-Riemannian structure induces a well defined Riemannian structure on $G/K$, which we can identify with~$\sfSigma_n(\rho)$. We remark that $\lambda \sfSigma_n(\rho)$ is isometric to $\sfSigma_n( \rho/\lambda^2 )$ for any $\lambda > 0$.

It is not clear what the analogue of sectional curvature should be in sub-Riemannian geometry. We will therefore also present the following alternative description of Riemannian model spaces. Let $(M, g)$ be a connected, simply connected Riemannian manifold. Then $(M,g)$ is a model space if and only if for every $(x,y) \in M \times M$ and every $q \in \Ort(T_xM, T_yM)$ there exists an isometry $\varphi$ with $\varphi_* |T_xM = q$, see e.g. \cite[Theorem~3.3]{KoNo63}. We will use this result as the basis of our definition for sub-Riemannian model spaces.
\begin{definition}
We say that $(M, D, g)$ is a sub-Riemannian model space if
\begin{enumerate}[\rm (i)]
\item $M$ is a connected, simply connected manifold,
\item $D$ is a bracket-generating subbundle,
\item for any linear isometry $q\in \Ort(D_x, D_y)$, $(x, y) \in M \times M$, there exists a smooth isometry $\varphi: M \to M$ such that $\varphi_*| D_x = q$.
\end{enumerate}
\end{definition}

\section{Sub-Riemannian model spaces and partial connections} \label{sec:Connection}
\subsection{Sub-Riemannian homogeneous spaces} \label{sec:HomSpaces}
We say that a sub-Riemannian manifold $(M, D,g)$ with $D$ bracket-generating is a \emph{homogeneous space} if for a pair of points $x,y \in M$ there exists an isometry satisfying $\varphi(x) = y$. Since the growth vector $\underline{n}(x)$ is determined by the metric $d_M$ in a neighborhood of $x$, we know the growth vector is constant on a homogeneous space. Hence, $D$ is equiregular and all isometries of $M$ are smooth maps that form a Lie group. We denote this Lie group $\Isom(M)$. We will use an approach inspired by \cite{FaGo95} to study these spaces.

Assume that $D$ is of step~$r$. Write $G = \Isom(M)$ and, for any $x \in M$, define its stabilizer group
$$G_x = \left\{ \varphi(x) = x \, : \, \varphi \in G \right\}.$$
By our assumptions, the subgroups $G_x$ are all conjugate and they are compact by \cite[Corollary~5.6]{GaKe03}. For a chosen point $x_0 \in M$, define $K = G_{x_0}$. Consider the $K$-principal bundle
$$K \to G \stackrel{\pi}{\to} M \cong G/K,$$
where $\pi(\varphi) = \varphi(x_0)$. Let $\mathfrak{g}$ and $\mathfrak{k}$ be the Lie algebras of respectively~$G$ and~$K$. Since $K$ is compact, any invariant subspace of a representation of~$K$ has an invariant complement as well. It follows that $\mathfrak{g} = \mathfrak{p} \oplus \mathfrak{k}$ where $\mathfrak{p}$ is some $K$-invariant subspace.

\begin{lemma}  \label{lemma:Action}
Let $\mathfrak{g} = \mathfrak{p}\oplus \mathfrak{k}$ be a decomposition of $\mathfrak{g}$ into $K$-invariant subspaces. For any $v \in T_{x_0}M$, let $A_v \in \mathfrak{p}$ be the unique element satisfying $\pi_* A_v = v$. 
Then
$$\Ad(\varphi) A_v = A_{\varphi_* v}, \qquad \text{for any $\varphi \in K$}.$$
In particular, $\mathfrak{p}^j := \left\{ A \in \mathfrak{p} \, : \, \pi_* A \in D_{x_0}^j \right\}$ is $K$-invariant for $j =1 ,\dots, r$.
\end{lemma}
\begin{proof}
For any $v \in T_{x_0} M$, if we write $\phi(t) = \exp t A_v$, then for any $\varphi \in G$,
\begin{equation} \label{Gaction} \pi_* \Ad(\varphi) A_v = \varphi_* \left. \frac{d}{dt} \left(\phi(t) \circ \varphi^{-1}\right)(x_0) \right|_{t = 0}. \end{equation}
Inserting $\varphi \in K$ in \eqref{Gaction}, we get that $\pi_* \Ad(\varphi) A_v = \varphi_* v$. Since $\Ad(\varphi)$ preserves~$\mathfrak{p}$, it follows that $\Ad(\varphi) A_v = A_{\varphi_* v}$. Hence, each $\mathfrak{p}^j$ is preserved under the action of $K$ by Remark~\ref{re:Dj}.
\end{proof}

\begin{remark} \label{re:orientable}
For any homogeneous space, we must have $D$ orientable. To see this, write $\rank D = n$ and let $x_0$ be an arbitrary point. Consider any $\mu \in \wedge^n D_{x_0}$ with $|\mu|_g =1$. For any isometry $\varphi$ on $M$, we have $| \varphi_* \mu | = 1$ as well. In particular, the image of the identity component of $\Isom(M)$ under the map $\varphi \mapsto \varphi_* \mu$ gives us a well-defined non-vanishing section of $\wedge^n D$.
\end{remark}

\subsection{Canonical partial connection on model spaces} \label{sec:Canonical}
In this section we will show that any sub-Riemannian model space has a canonical partial connection on~$D$, coinciding with the Levi-Civita connection when $D = TM$. Recall that if $\pi: E \to M$ is a vector bundle and $D$ is a subbundle of~$TM$, then a partial connection $\nabla$ on $E$ in the direction of $D$ is a map $\nabla: \Gamma(E) \to \Gamma(D^* \otimes E)$ such that $\nabla \phi e = d\phi|D \otimes e + \phi\nabla e$ for any $\phi \in C^\infty(M)$ and $e \in \Gamma(E)$. As usual, we write $(\nabla e)(v) = \nabla_v e$ for $v \in D$. It is important to note that such partial connections defines parallel transport of elements in $E$ along $D$-horizontal curves. For more on partial connections, see \cite{FGR97}.

Let $(M,D,g)$ be a sub-Riemannian manifold and let $\nabla$ be a partial connection on $D$ in the direction of $D$. We say that $\nabla$ is \emph{compatible} with $g$ if $ \langle Y, Z \rangle_g = \langle \nabla_X Y,Z \rangle_g + \langle Y, \nabla_X Z\rangle_g$ for any $X, Y, Z \in \Gamma(D)$. Equivalently, a partial connection is compatible with $g$ if it takes orthonormal frames of $D$ to orthonormal frames under parallel transport. If $G$ is a group whose elements are diffeomorphisms of $M$ satisfying $\varphi_* (D) \subseteq D$ for every $\varphi \in G$, then we say that $\nabla$ is \emph{invariant} under $G$ if for every
for every $\varphi \in G$ and $X, Y \in \Gamma(D)$,
$$\nabla_{\Ad(\varphi) X} \Ad(\varphi) Y = \Ad(\varphi) \nabla_X Y.$$
Recall that we have defined $\Ad(\varphi) X(x) =  \varphi_* X(\varphi^{-1}(x))$.
\begin{proposition} \label{prop:Connection}
Let $(M, D, g)$ be a sub-Riemannian model space with isometry group $\Isom(M)$. Then there exists a unique partial connection on $D$ in the direction of $D$ that is compatible with $g$ and invariant under $\Isom(M)$. We call this the canonical partial connection on $D$.
\end{proposition}

Before stating the proof of the above result, we would like to review some theory of partial connections and orthonormal frame bundles.

On a principal bundle $K \to P \stackrel{\pi}{\to} M$, where $K$ has Lie algebra $\mathfrak{k}$, a partial connection form $\omega$ in the direction of $D \subseteq TM$ is a partial one-from $\omega: (\pi_*)^{-1} D \to \mathfrak{k}$ satisfying the conditions
$$\textstyle \omega(\frac{d}{dt} p \cdot e^{tA} |_{t=0} ) = A, \qquad \omega(v \cdot a) = \Ad(a^{-1})\omega(v), \quad A\in \mathfrak{k}, p \in P, a \in K, v \in TP.$$
A partial connection will hence give us decomposition $(\pi_*)^{-1} D = \ker \omega \oplus \ker \pi_* =: \calE \oplus \ker \pi_*$ where $\calE$ is a subbundle invariant under the action of $K$.

Consider the orthonormal frame bundle $F^{\Ort}(D)$, i.e., the $\Ort(n)$-principal bundle
$$\Ort(n) \to F^{\Ort}(D) \stackrel{\pi}{\to} M, \qquad n = \mathrm{rank} \, D,$$
such that
$$F^{\Ort}(D)_x =  \Ort( \mathbb{R}^n , D_x),$$
with the standard inner product on $\mathbb{R}^n$ and with $\Ort(n)$ acting on the right by composition. If we let $e_1, \dots, e_n$ denote the standard basis of $\mathbb{R}^n$, we can identify $f \in \Ort(D)_x$ with the orthonormal frame $\{ f(e_1), \dots, f(e_n) \}$ of~$D_x$. For any partial connection $\nabla$ compatible with $g$, we can define a corresponding partial connection form $\omega$ on $F^{\Ort}(D)$ by defining $\calE = \ker \omega$ as the derivatives of $\nabla$-parallel frames.

We have an action of $G= \Isom(M)$ on $F^{\Ort}(D)$ by $\varphi \cdot f := \varphi_* f$. Using Theorem~\ref{th:Isom} and the assumption that $M$ is a model space, we know that this action if free and transitive. By choosing a reference frame $f_0 \in F^{\Ort}(D)_{x_0}$, $x_0 \in M$ and writing $K = G_{x_0}$, we can identify $G$ with $F^{\Ort}(D)$ through the map $\varphi \mapsto \varphi \cdot f_0$. With this identification, the subbundle $\calE$ is right invariant under~$K$. Furthermore, if the partial connection $\nabla$ also satisfies $\nabla_{\Ad(\varphi) X} \Ad(\varphi) Y = \Ad(\varphi) \nabla_X Y$ for every $X ,Y \in \Gamma(D)$ and $\varphi \in G$, then $\calE$ is left invariant under $G$ as well. Restricting this subbundle to the identity $1$ of $G$ and considering the Lie algebra $\mathfrak{g}$, we get a subspace $\mathfrak{p}^1 = \calE_1 \subseteq \mathfrak{g}$ which is $K$-invariant and satisfies $\mathfrak{p}^1 \cap (\ker \pi_*) = 0$ and $\pi_* \mathfrak{p}^1 = D_{x_0}$.

\begin{proof}[Proof of Proposition~\ref{prop:Connection}]
Since $K$ is compact, its Lie algebra $\mathfrak{k}$ has a $K$-invariant complement $\mathfrak{p}$ in $\mathfrak{g}$, and by defining $\mathfrak{p}^1$ as in Lemma~\ref{lemma:Action}, we have a desired partial connection. This proves existence.

Next, we prove uniqueness. Let $\mathfrak{p}^1$ and $\mathfrak{q}^1$ be two $K$-invariant subspaces, transverse to $\mathfrak{k}$ and satisfying $\pi_* \mathfrak{p}^1 = \pi_* \mathfrak{q}^1 = D_{x_0}$. For any $v \in D_{x_0}$, write $A_v$ and $\hat A_v$ for the respective elements in $\mathfrak{p}^1$ and $\mathfrak{q}^1$ that project to $v$. Consider the map
\begin{equation} \label{zeta} \zeta: \mathfrak{p}^1 \to \mathfrak{k}, \qquad \zeta(A_v) = A_v - \hat A_v.\end{equation}
Then for any $a \in K$, $\Ad(a) \zeta = \zeta \Ad(a)$. We emphasize for the reader that by our assumption that $M$ is a model space, we have
$$K \cong \Ort(n),$$
making $G$ of dimension $\dim M + \frac{n(n-1)}{2}$. We observe that the $\Ad$-action of~$K$ on~$\mathfrak{k}$ is isomorphic to the usual adjoint representation of~$\Ort(n)$ on~$\ort(n)$. Furthermore, $K$ acts on~$\mathfrak{p}^1$ as in Lemma~\ref{lemma:Action}, which is isomorphic to the usual representation of $\Ort(n)$ on $\mathbb{R}^n$. These representations are irreducible and never isomorphic, see Lemma~\ref{lemma:Representation}, Appendix, so $\zeta = 0$. 
\end{proof}

\begin{remark} \label{re:Compare}
We can use the canonical partial connection to determine when two model spaces are isometric. For $j=1,2$, let $(M^{(j)}, D^{(j)}, g^{(j)})$ be a sub-Riemannian model space with isometry group $G^{(j)} = \Isom(M^{(j)})$. Let $\nabla^{(j)}$ be the canonical partial connection on $D^{(j)}$. Assume that there is an isometry $\Phi: M^{(1)} \to M^{(2)}$. Choose an arbitrary point $x_0 \in M^{(1)}$, define $y_0 = \Phi(x_0) \in M^{(2)}$ and
$$K^{(1)} = (G^{(1)})_{x_0}, \qquad K^{(2)} =  (G^{(2)})_{y_0}.$$
Let $\mathfrak{k}^{(j)}$ be the Lie algebra of $K^{(j)}$.
The map $\Phi$ induces a group isomorphism $\bar{\Phi}: G^{(1)} \to G^{(2)}$ defined by
$$\bar{\Phi}(\varphi) = \Phi \circ \varphi \circ \Phi^{-1}, \qquad \varphi \in G^{(1)}.$$
We remark that $\bar{\Phi} (K^{(1)}) = K^{(2)}$. Furthermore, let $\mathfrak{g}^{(j)}$ denote the Lie algebra of $G^{(j)}$. Use a choice of orthonormal frame of respectively $D^{(1)}_{x_0}$ and $D^{(2)}_{y_0}$ to identify $F^{\Ort}(D^{(j)})$ with $G^{(j)}$ and let $\mathfrak{p}^{(j)}$ be subspace of $\mathfrak{g}^{(j)}$ corresponding to the canonical partial connection $\nabla^{(j)}$. Define an inner product on $\mathfrak{p}^{(j)}$ by pulling back the inner product on respectively $D_{x_0}^{(1)}$ and~$D_{y_0}^{(2)}$. Since
$$\nabla_{\Ad(\Phi) X}^{(2)} \Ad(\Phi) Y = \Ad(\Phi) \nabla_X^{(1)} Y$$
and since $\Phi$ is an isometry, we must have that $\Psi := \bar{\Phi}_{*,\id}: \mathfrak{g}^{(1)} \to \mathfrak{g}^{(2)}$ maps~$\mathfrak{p}^{(1)}$ isometrically onto $\mathfrak{p}^{(2)}$. In conclusion, if $M^{(1)}$ is isometric to $M^{(2)}$, then there is a Lie algebra isomorphism $\Psi: \mathfrak{g}^{(1)} \to \mathfrak{g}^{(2)}$  mapping $\mathfrak{k}^{(1)}$ onto~$\mathfrak{k}^{(2)}$ and $\mathfrak{p}^{(1)}$ onto $\mathfrak{p}^{(2)}$ isometrically.
\end{remark}

\subsection{Holonomy of the canonical partial connections} \label{sec:Holonomy}
Since all model spaces have canonical partial connections, we will need to look at the holonomy and curvature of such a connection. We will use our material from~\cite{CGJK15}.

Let $\pi: E \to M$ be a vector bundle and let $\nabla$ be a partial connection on $E$ in the direction of $D$. Assume that $D$ is bracket-generating and equiregular of step~$r$. For any horizontal curve $\gamma: [0,1] \to M$, write $\ptr_t^\gamma: E_{\gamma(0)} \to E_{\gamma(t)}$ for the parallel transport along $\gamma$ with respect to $\nabla$. Then \emph{the horizontal holonomy group} of $\nabla$ at $x$ is given by
$$\Hol^\nabla(x) = \left\{ \ptr_1^\gamma \in \GL(E_x) \, : \, \begin{array}{c} \text{$\gamma: [0,1] \to M$ is a} \\ \text{horizontal loop based at $x$} \end{array} \right\},$$
where $\GL(E_x)$ denotes invertible linear maps of $E_x$.
The group $\Hol^{\nabla}(x)$ is a finite dimensional Lie group and it is connected if $M$ is simply connected. In order to determine this holonomy group, we need to introduce \emph{selectors}.

Consider the flag $0 = D^0 \subseteq D = D^1 \subseteq D^2 \subseteq \cdots \subseteq D^r= TM$ defined as in \eqref{flag}. For any $k =0,1,\dots,r$, let $\Ann(D^k)$ denote the subbundle of~$T^*M$ consisting of all covectors vanishing on $D^k$. A selector of $D$ is a two-vector valued one-form $\chi: TM \to \wedge^2 TM$, satisfying the following for every $k =1, \dots, r$:
\begin{enumerate}[\rm (i)]
\item $\chi(D^k) \subseteq \wedge^2 D^{k-1},$
\item for any $\alpha \in \Gamma(\Ann D^k)$ and $w \in D^{k+1}$, we have
$$\alpha(w) = - d\alpha(\chi(w)).$$
\end{enumerate}
Any equiregular subbundle have at least one selector, but this selector is in general not unique. However, each selector gives a unique way of extending partial connections to connections.

\begin{theorem} \label{th:CGJK}
For any partial connection $\nabla$ on $\pi: E \to M$ in the direction of $D$ and selector $\chi$ of $D$, there exists a unique affine connection $\bnabla = \nabla^\chi$ on $E$ such that $\bnabla_{|D} = \nabla$ and such that
$$R^{\bnabla}(\chi(\cdot)) = 0.$$
Furthermore, for any $x \in M$,
$$\Hol^{\nabla}(x) = \Hol^{\bnabla}(x).$$
\end{theorem}
In the special case when there exists an affine connection $\bnabla$ such that $\bnabla_{|D} = \nabla$ and such that $R^{\bnabla} = 0$, it follows from Theorem~\ref{th:CGJK} that $\bnabla = \nabla^\chi$ for any selector. Furthermore, if $M$ is simply connected, then $\Hol^\nabla(x) = \id_{E_x}$ from the Ambrose-Singer theorem.

For the canonical partial connection on model spaces there are only two possible cases for the horizontal holomy group.
\begin{proposition} \label{prop:HolCC}
Let $(M, D, g)$ be a sub-Riemannain model space with canonical partial connection $\nabla$. Then either $\Hol^\nabla(x)= \id_{D_x}$ for every $x \in M$ or $\Hol^\nabla(x) = \SO(D_x)$ for every $x \in M$.
\end{proposition}
\begin{proof}
Let $\calE$ be the $\Ort(n)$-invariant subbundle of $TF^{\Ort}(D)$ corresponding to~$\nabla$. For any $f_0 \in F^{\Ort}(D)$, define
$$\mathcal{O}_{f_0} = \{ f \in F^{\Ort}(D) \, : \, \text{there is an $\calE$-horizontal curve connecting $f_0$ and $f$} \}.$$
By the Orbit Theorem \cite[Theorem~5.1]{AgSa04} these are imbedded submanifolds and for every $f \in \mathcal{O}_{f_0}$ and every $j \geq 1$
$$\{ Y(f) \, : \, Y \in \underline{\calE}^j\} \subseteq T_f\mathcal{O}_{f_0}.$$
Here, $\underline{\calE}^j$ is defined as in \eqref{Dunderline}. Furthermore, by definition, for any $f\in F^{\Ort}(D)_x$ we know that $\Hol^{\nabla}(x)$ is isomorphic to
$$ \{ a \in \Ort(n) \, : \,  f \cdot a \in \calO_{f} \}.$$

 Let $x_0 \in M$ be a chosen point. Choose $f_0 \in F^{\Ort}(D)_{x_0}$ and use this frame to identify $F^{\Ort}(D)$ with $G = \Isom(M)$. Define $K = G_{x_0}$ and let $\pi:G \to M$ be the map $\pi(\varphi) = \varphi(x_0)$. Let $\mathfrak{k}$ and $\mathfrak{g}$ denote the Lie algebra of respectively $K$ and~$G$.  Let $\mathfrak{p}^1$ be the $K$-invariant subspace of $\mathfrak{g}$ of rank $n$ such that $\calE_{\varphi} = \varphi \cdot \mathfrak{p}^1$, $\varphi \in G$. Let $\hat{ \mathfrak{p}}$ be the sub-algebra generated by $\mathfrak{p}^1$ which is also $K$-invariant. Then $\hat{\mathfrak{p}} \cap \mathfrak{k}$ is also $K$-invariant. Since the action of $K$ on $\mathfrak{k}$ is isomorphic to the adjoint representation of $\Ort(n)$ on $\ort(n)$, which is irreducible by Lemma~\ref{lemma:Representation}, we either have that $\mathfrak{g} = \hat{\mathfrak{p}} \oplus \mathfrak{k}$ or $\mathfrak{g} = \hat{\mathfrak{p}}$. It follows that either $
\calE$ is bracket-generating or integrable. If $\calE$ is bracket-generating, then $\calO_{f}$ is the connected component of $f$ in $F^{\Ort}(D)$, so $\Hol^{\nabla}(x) = \SO(D_x)$. If $\calE$ is integrable, $\Hol^\nabla(x)$ must be discrete, and since $M$ is simply connected, we have $\Hol^{\nabla}(x) = \id_{D_x}$.
\end{proof}

\subsection{Holonomy and Lie group structure} \label{sec:Lie}
In Riemannian geometry, the only model spaces with a Lie group structure for which the metric is invariant are the euclidean spaces $\sfSigma_n(0)$ and the 3-dimensional spheres $\sfSigma_3(\rho)$, $\rho > 0$. These two examples differ, however, in that all isometries preserving the identity of $\sfSigma_n(0)$ are Lie algebra automorphisms. This property fails for $\sfSigma_3(\rho)$ since the inversion map is an isometry but the group is not abelian. In this section, we will consider the relationship between trivial holonomy and group structure and leave examples similar to thee-dimensional spheres for the next section.

\begin{proposition} \label{prop:Lie}
Let $(M,D,g)$ be a model space and write $G = \Isom(M)$. Let $x_0$ be an arbitrary point. Then $\Hol^{\nabla}(x_0) =\id_{D_{x_0}}$ if and only if there exists a Lie group structure on $M$ such that $(D,g)$ is left invariant, $x_0$ is the identity and every $\varphi \in G_{x_0}$ is a Lie group automorphism.
\end{proposition}

\begin{proof}
Assume that $M$ has a Lie group structure with identity $1 = x_0 \in M$, that $(D,g)$ is left invariant and that every $\varphi \in G_{x_0}$ is a Lie group automorphism. Define $\nabla^l$ on $D$ such that all left invariant vector fields are parallel. This obviously satisfies $\Hol^{\nabla^l}(x_0) = \id_{D_{x_0}}$ and is invariant under $G$.

Conversely, suppose that $\Hol^{\nabla}(x_0) =\id_{D_{x_0}}$ for a connection invariant under $G$. By the proof of Proposition~\ref{prop:HolCC}, the Lie algebra $\mathfrak{g}$ of $G$ has decomposition $\mathfrak{g} = \hat{\mathfrak{p}} \oplus \mathfrak{k}$ into $K$-invariant vector spaces with $[\hat{\mathfrak{p}}, \hat{\mathfrak{p}}] \subseteq \hat{\mathfrak{p}}$ and where $\mathfrak{k}$ is the Lie algebra of $K = G_{x_0}$. Hence, every element in $K$ is a Lie group automorphism and $M$ is the Lie group of $\hat {\mathfrak{p}}$.
\end{proof}

\subsection{Horizontal bundles of rank~3 and group structure}
If we want to look at model spaces with Lie group structures where only the orientation-preserving isometries are group isomorphisms, such as $\sfSigma_3(\rho)$, $\rho >0$, we need some special considerations for the case of $\rank D =3$, as this condition allows one to define a cross product. Let $\mu \in \Gamma(\wedge^3 D)$ be any section satisfying $|\mu|_g=1$, which exists by Remark~\ref{re:orientable}. For any pair $v_i,w_i\in D_x$, $x \in M$, $i=1,2$, define $w_1 \times w_2 \in D_x$ by
$$v_1 \wedge v_2 \wedge (w_1 \times w_2) = \langle v_1 \wedge v_2, w_1 \wedge w_2 \rangle_g \mu(x).$$
Our choice of $\mu$ only determines the sign of the cross product. In particular, the collection of maps $(w_1, w_2) \mapsto c w_1 \times w_2$, $c \in \mathbb{R}$ does not depend on the choice of $\mu$.

\begin{proposition} \label{prop:OrientationPre}
Let $(M, D, g)$ be a sub-Riemannian model space with canonical partial connection $\nabla$. Let $G^0$ denote the identity component of $\Isom(M)$, i.e. isometries that preserves orientation. Let $x_0$ be an arbitrary point and let $G_{x_0}^0$ be the group of elements in $G^0$ fixing $x_0$. 
\begin{enumerate}[\rm (a)]
\item If $\rank D \neq 3$, then the only partial connection that is compatible with $g$ and invariant under $G^0$ is $\nabla$. Furthermore, $M$ has a Lie group structure such that $(g,D)$ is left invariant, $x_0$ is the identity and all elements in $G_{x_0}^0$ are Lie algebra isometries if and only if $\Hol^\nabla(x_0) = \id_{D_{x_0}}$.
\item If $\rank D = 3$, then there is a one-parameter family $\nabla^c$, $c \in \mathbb{R}$, of  compatible partial connections invariant under $G^0$. These are given by
$$\nabla^c_X Y = \nabla_X Y + c X \times Y, \qquad X,Y \in \Gamma(D),$$
 where $\times$ denotes any cross product on $D$. Furthermore, $M$ has a Lie group structure such that $(g,D)$ is left invariant, $x_0$ is the identity and all elements in $G_{x_0}^0$ are Lie algebra isometries if and only if $\Hol^{\nabla^c}(x_0) = \id_{D_{x_0}}$ for some $c \in \mathbb{R}$.
\end{enumerate}
\end{proposition}

\begin{proof}[Proof of Proposition~\ref{prop:OrientationPre}]
We begin by choosing an orientation of $D$, and consider the oriented orthonormal frame bundle
$$\SO(n) \to F^{\SO}(D) \to M,$$
where $F^{\SO}(D)_x = \SO(\mathbb{R}^n, D_x)$. Similar relations for invariant partial connections hold on $F^{\SO}(D)$ with $G^0$ in the place of $\Isom(M)$. Hence, if we modify the proof of Proposition~\ref{prop:Connection}, by replacing $F^{\Ort}(D)$, $G = \Isom(M)$ and $K = G_{x_0}$ with $F^{\SO}(D)$, $G^0 = \Isom_0(M)$ and $K^0= G^0_{x_0}$, respectively, then we can identify maps $\zeta$ as in \eqref{zeta} with morphisms of representations of $\SO(n)$ on respectively $\mathbb{R}^n$ and $\mathfrak{\ort}(n)$. We will have $\zeta = 0$ whenever $n \neq 3$, since there are still no morphisms between the standard and the adjoint representations of $\SO(n)$, see Remark~\ref{re:Hodge}, Appendix, for details. Following the same remark for the case $n =3$, the standard representation and the adjoint representation of $\SO(3)$ are isomorphic through the map $x \wedge y \mapsto x \times y$.

The remainder of the proof is identical to Proposition~\ref{prop:Lie}.
\end{proof}

\section{Metric tangent cones and Carnot-groups} \label{sec:Carnot}
\subsection{Carnot model spaces} \label{sec:Carnot1}
Let $(N, D, g)$ be \emph{a Carnot group} of step~$r$. In other words, $N$ is a nilpotent, connected, simply connected Lie group whose Lie algebra $\mathfrak{n}$ has a stratification $\mathfrak{n} = \mathfrak{n}_1 \oplus \cdots \oplus \mathfrak{n}_r$ satisfying
\begin{equation} \label{stratification} [\mathfrak{n}_1, \mathfrak{n}_j ] = \left\{ \begin{array}{ll} \mathfrak{n}_{j+1} & \text{for $1 \leq j < r$}, \\
0 & \text{for $j = r$}.
  \end{array} \right. \end{equation}
The sub-Riemannian structure $(D, g)$ is defined by left translation of $\mathfrak{n}_1$ with some inner product. Because of this left invariance, a Carnot group is a model space if and only if for every $q \in \Ort(\mathfrak{n}_1)$ there exists an isometry $\varphi_q$ satisfying $\varphi_q(1)= 1$ and ${\varphi_q}_*|\mathfrak{n}_1 =q$. By \cite[Theorem~1.1]{LDOt16} all such isometries are also Lie group automorphisms. We will refer to Carnot groups that are also model spaces as \emph{Carnot model spaces}.

The free nilpotent Lie groups form a class of Carnot model spaces. For a finite set $\{A_1, \dots, A_{n} \}$ let $\mathfrak{f} = \mathsf{f}[ A_1, \dots, A_{n}; r ]$ denote the corresponding free nilpotent Lie algebra of step~$r$, i.e. the free Lie algebra with generators $A_1, \dots, A_{n}$ divided out by the ideal of all brackets of length $\geq r+1$. This algebra has a natural grading $\mathfrak{f} = \mathfrak{f}_1 \oplus \cdots \oplus \mathfrak{f}_r$, where $\mathfrak{f}_1$ is spanned by elements $A_1, \dots, A_{n}$ and for $j \geq 2,$
$$\mathfrak{f}_{j} = \spn \left\{ [A_{i_1}, [ A_{i_2}, [\cdots, [A_{i_{j-1}}, A_{i_j} ]] \cdots ]] \, : \, 1 \leq i_k \leq n \right\}.$$

All such Lie algebras generated by $n$ elements will be isomorphic. For this reason, we will write $\mathsf{f}[A_1, \dots, A_n;r]$ simply as $\mathfrak{f} = \mathsf{f}_{n,r}$. Remark that for any invertible linear map $q \in \GL(\mathfrak{f}_1)$, there is a corresponding Lie algebra automorphism $\psi(q)$ of $\mathfrak{f}$ that preserves the grading and satisfies $\psi(q)|\mathfrak{f}_1 = q$ since the Lie algebras generated by $A_1, \dots, A_n$ and $qA_1, \dots, q A_n$ are isomorphic.

Define an inner product on $\mathfrak{f}_1$ such that the elements $A_1, \dots, A_n$ form an orthonormal basis. Write $\sfF_{n,r}$ for the corresponding connected, simply connected Lie group of $\mathsf{f}_{n,r}$ with a sub-Riemannian structure $(E, h)$ given by left translation of $\mathfrak{f}_1$ and its inner product. Write $\varphi_q$ for the Lie group automorphism corresponding to~$\psi(q)$.  Then $\varphi_q$ is an isometry whenever $q \in \Ort(\mathfrak{f}_1)$ and hence the Carnot group $(\sfF_{n,r},E,h)$ is a sub-Riemannian model space.

Next, let $(N,D,g)$ be any Carnot group with Lie algebra~$\mathfrak{n}$. Let $r$ denote its step and define $n = \rank D$. Let $A_1, \dotsm, A_n$ be an orthonormal basis of $\mathfrak{n}_1$ and consider $\sff_{n,r} = \mathfrak{f}_1 \oplus \cdots \oplus \mathfrak{f}_r$ as the free nilpotent Lie algebra of step $r$ generated by these elements. If we let $(\sfF_{n,r}, E, h)$ be the corresponding sub-Riemannian model space, then the Lie group $N$ can be considered as $\sfF_{n,r}$ divided out by some additional relations. Hence, there is a surjective group homomorphism
$$\Phi: \sfF_{n,r} \to N,$$
such that $\Phi_* |_{E_x} : E_x \to D_{\Phi(x)}$ is a linear isometry for every $x$. In particular, $\Phi_{*,1} : \sff_{n,r} \to \mathfrak{n}_1$ is a Lie algebra homomorphism and we have a canonical identification of $\mathfrak{n}_1$ and $\mathfrak{f}_1$ by construction. This allows us to lift any Lie algebra isomorphism $\tilde \psi$ to $\sff_{n,r}$, since, from the stratification of the Lie algebra, every such map is uniquely determined by $q = \tilde \psi | \mathfrak{n}_1$ and we must have
$$\Phi_{*,1} \circ \psi(q) = \tilde \psi \circ \Phi_{*,1}.$$
Conversely, any isomorphism $\psi(q)$ of $\sff_{n,r}$ corresponds to an isometry of $\mathfrak{n}$ only if it preserves the ideal $\mathfrak{a} = \ker \Phi_{*,1}$. In order for $N$ to be a model space, this ideal has to be preserved by every $q \in \Ort(\mathfrak{f}_1)$.

In conclusion, let $\Aut \mathfrak{g}$ denote the group of Lie algebra automorphism of $\mathfrak{g}$. Then the study of Carnot model spaces reduces to studying ideals $\mathfrak{a}$ of $\sff_{n,r}$ that are also sub-representations of
$$\psi: \Ort(\mathfrak{f_1}) \to \Aut \mathsf{f}_{n,r}, \quad q \mapsto \psi(q).$$

\begin{example} \label{ex:Step2}
Let us first consider $\mathfrak{f} = \mathsf{f}_{n,2}$. This can be seen as the vector space $\mathbb{R}^n \oplus \wedge^2 \mathbb{R}^2$, where $\wedge^2 \mathbb{R}^n$ is the center and
$$[x,y] = x \wedge y \qquad \text{for any $x,y \in \mathbb{R}^n$.}$$
The induced action of $\psi(q)$ on $\sff_{n,2}$ gives us the usual representation of $\Ort(n)$ on $\mathfrak{f}_2 = \wedge^2 \mathbb{R}^n$. By Lemma~\ref{lemma:Representation} this representation is irreducible and so $\sfF_{n,2}$ is the only Carnot group of step~$2$ that is also a model space. 
\end{example}

\begin{example} \label{ex:FirstC}
Consider $\mathfrak{f} = \mathsf{f}_{n,3}= \mathfrak{f}_1 \oplus \mathfrak{f}_2 \oplus \mathfrak{f}_3$. Identify $\mathfrak{f}_1$ and $\mathfrak{f}_2$ with the vector spaces $\mathbb{R}^n$ and $\ort(n)$, such that for $x, y \in \mathfrak{f}_1 = \mathbb{R}^n$,
$$[x,y] = x \wedge y \in \mathfrak{f}_2 = \ort(n).$$
For any $q \in \Ort(\mathfrak{f}_1)$, the action of $\psi(q)$ on $\mathfrak{f}_1$ and $\mathfrak{f}_2$ is given by
\begin{equation} \label{ActionPsiQ} \psi(q) x = qx, \qquad \psi(q) A = \Ad(q) A, \qquad x \in \mathfrak{f}_1, A \in \mathfrak{f}_2.\end{equation}
Define $\mathfrak{a} \subseteq \mathfrak{f}_3$ by
$$\mathfrak{a} = \left\{ \sum_{k=1}^j [A_{k}, x_{k}] \, :  \, \begin{array}{c} j \geq 1, x_{k} \in \mathfrak{f}_1, A_{k} \in \mathfrak{f}_2 \\
 \sum_{k=1}^j A_{k} x_{k}= 0 \end{array} \right\}.$$
This ideal is invariant under $\psi(q)$ by \eqref{ActionPsiQ}, strictly contained in $\mathfrak{f}_3$ and is nonzero if $n \geq 3$. Write $\sfc_{n,3} := \mathsf{f}_{n,3}/\mathfrak{a}$. We can then identify $\sfc_{n,3}$ with the vector space $\mathbb{R}^n \oplus \ort(n) \oplus \mathbb{R}^n$ with non-vanishing brackets
$$[(x, 0, 0), (y,0,0) ] = (0, x \wedge y, 0), \quad [(0, A, 0), (x, 0,0)] = (0, 0, Ax),$$
for $x,y \in \mathbb{R}^n$ and $A \in \ort(n)$. If we let $\sfC_{n,3}$ denote the corresponding connected, simply connected Lie group and give it a sub-Riemannian structure $(D,g)$ by left translation of elements $(x, 0, 0)$, $x \in \mathbb{R}^n$ with the standard euclidean metric. Then $\sfC_{n,3}$ is a sub-Riemannian model space.
\end{example}

\begin{example} \label{ex:Rol}
We will elaborate on Example~\ref{ex:FirstC}. For any $n \geq 2$ and $r \geq 1$, define the stratified Lie algebra $\mathfrak{c} = \sfc_{n,r} = \mathfrak{c}_1 \oplus \cdots \oplus \mathfrak{c}_r$ as follows. Write $\mathfrak{a}_j = \mathfrak{c}_{2j -1}$ and $\mathfrak{b}_j = \mathfrak{c}_{2j}$. Each $\mathfrak{a}_j$ equals $\mathbb{R}^n$ as a vector space, while each $\mathfrak{b}_j$ equals $\mathfrak{o}(n)$ as a vector space. For each $x \in \mathbb{R}^n$, write $x^{(j)}$ for the element in $\mathfrak{c}$ whose component in $\mathfrak{a}_j$ is equal to $x$, while all other components are zero. For any $A \in \ort(n)$, define $A^{(j)} \in \mathfrak{b}_j \subseteq \mathfrak{c}$ analogously. If $2j-1 > r$ (respectively $2j >r$), we define $x^{(j)} = 0$ (respectively $A^{(j)} = 0$). The Lie brackets on $\mathfrak{c}$ are then determined by relations
\begin{equation} \label{RolNil}
\left\{ \begin{array}{rcl}
{[x^{(i)}, y^{(j)}]} & =& (x \wedge y)^{(i+j-1)}, \\
{[A^{(i)} , x^{(j)} ]} & = & (Ax)^{(i+j)} , \\
{[A^{(i)}, B^{(j)} ]} & = & [A, B]^{(i+j)},
\end{array} \right.
\end{equation}
for any $x, y \in \mathbb{R}^n$ and $A,B \in \ort(n)$. These relations make $\mathfrak{c}$ a nilpotent, stratified Lie algebra of step $r$. Give $\mathfrak{c}_1$ an inner product corresponding to the standard inner product on $\mathbb{R}^n$. The linear maps $\psi(q)$, $q \in \Ort(n)$, given by
$$\psi(q) x^{(j)} = (qx)^{(j)}, \qquad \psi(q) A^{(j)} = (\Ad(q) A)^{(j)},$$
are Lie algebra automorphisms that preserve the grading and inner product on $\mathbb{R}^n$. Hence, the corresponding connected, simply connected group is a sub-Riemannian model space, which we denote by $\sfC_{n,r}$. Remark that $\sfC_{n,2}$ is isometric to $\sfF_{n,2}$, while $\sfC_{n,3}$ is only isometric to $\sfF_{n,3}$ for the special case of $n = 2$. We use this notation since $\sfC_{2,3}$ is the Cartan group.
\end{example}

\begin{example} \label{ex:list}
\begin{enumerate}[\rm (a)]
\item The Engel group is the Carnot group whose stratified Lie algebra $\mathfrak{n} = \mathfrak{n}_1 \oplus \mathfrak{n}_2 \oplus \mathfrak{n}_3$ is given by
$$\mathfrak{n}_1 = \spn\{ X_1, X_2\}, \quad \mathfrak{n}_2 = \spn\{ Z_1\}, \quad \mathfrak{n}_3 = \spn \{Z_2\}.$$
$$[X_1, X_2] = Z_1, \quad [X_1, Z_1] = Z_2,$$
$$[Z_2, Z_1] = [Z_2, X_1] = [Z_2, X_2] = [Z_1, X_2] = 0.$$
This is not a sub-Riemannian model space. We see that the isometry
$$q: \mathfrak{n}_1 \to~\mathfrak{n}_1, \qquad qX_1 = X_2, \qquad qX_2= - X_1,$$
does not correspond to any Lie algebra isomorphism.
\item Consider the $n$-th Heisenberg group $ H[n]$ whose stratified Lie algebra $\mathfrak{h}[n] = \mathfrak{h}_1 \oplus \mathfrak{h}_2$ satisfies relations
$$\mathfrak{h}_1 = \spn \{ X_i, Y_i \, :\, i=1, \dots, n\}, \quad \mathfrak{h}_2 = \spn \{ Z\},$$
$$0 = [Z, X_i] = [Z,Y_i] = [X_i, X_j] = [Y_i, Y_j], \qquad [X_i, Y_j] = \delta_{ij} Z.$$
It is simple to verify that $H[n]$ is a model space if and only if $n =2$. In the latter case, $H[2]$ is isometric to $\sfF_{2,2}$.
\end{enumerate}
\end{example}

\subsection{Model spaces and their tangent cones} \label{sec:Tangent}
Let $(M,D, g)$ be a sub-Riemannian manifold with a bracket-generating, equiregular subbundle $D$ and let $d_M$ be the Carnot-Carath\'eodory metric on $M$. Define subbundles $D^j$, $0 \leq j \leq r$ as in \eqref{flag} and use the convention $D^j = TM$ whenever $j > r$. For a fixed $x \in M$, define the following nilpotent Lie algebra $\nil(x)$.
\begin{enumerate}[\rm (i)]
\item As a vector space, $\nil(x) = \mathfrak{n}_1 \oplus \cdots \oplus \mathfrak{n}_r$ with $\mathfrak{n}_j := D^j_x/ D^{j-1}_x$.
\item If $A = [v] \in \mathfrak{n}^i$ and $B = [w] \in \mathfrak{n}^j$ denote the equivalent classes of $v \in D^i_x$ and $w \in D^j_x$, respectively, we define
$$[A, B] := [X, Y](x) \quad \text{mod } D^{i+j-1},$$
$$ \quad X \in \Gamma(D^i), Y \in \Gamma(D^j), \quad X(x) = v, \quad Y(x) = w.$$
\end{enumerate}
The subspace $\mathfrak{n}_1 = D^1_{x_0}$ comes equipped with an inner product. Let $\Nil(x)$ be the simply connected Lie group of $\nil(x)$ and define a sub-Riemannian structure $(E,h)$ on~$\Nil(x)$ by left translation of $\mathfrak{n}_1$ and its inner product. If $d_{\Nil(x)}$ is the Carnot-Carath\'eodory metric of $(E, h)$ then $(\Nil(x), d_{\Nil(x)}, 1)$ is the tangent cone of the metric space $(M, d_M)$ at $x$, see \cite{Bel96,LeD16}.

Let $\varphi: (M,D, g) \to (\tilde M, \tilde D, \tilde g)$ be a diffeomorphism map between two sub-Riemannian manifolds such that $\varphi_* D \subseteq \tilde D$. Then for any $x \in M$, there is a corresponding map
$$\Nil_x \varphi: \Nil(x) \to \Nil(\varphi(x)),$$
defined as the Lie group homomorphism such that $(\Nil_x \varphi)_{*,1}: \nil(x) \to \nil(\varphi(x))$ equals
$$(\Nil_x \varphi)_* [v] = [\varphi_* v], \qquad [v] \in D^j_{x}/D^{j-1}_x,$$
This is well defined since $\varphi_* D^j \subseteq \tilde D^j$ by Remark~\ref{re:Dj}.

\begin{proposition}
Let $(M, D, g)$ be a model space and $x\in M$ any point. Then the tangent cone $\Nil(x)$ is a model space as well. In particular, the growth vector of any sub-Riemannian model space equals that of a Carnot group model space.
\end{proposition}

\begin{proof}
$\Nil(x)$ is a model space if and only if for every $q \in \Ort(\mathfrak{n}_1) = \Ort(D_x)$ there exists a Lie group automorphism $\phi:\Nil(x) \to \Nil(x)$ with $\phi_* |\mathfrak{n}_1 = q$. For every $q \in \Ort(D_x)$, define $\varphi_q: M \to M$ as the isometry such that ${\varphi_q}_* |D_x  = q$. Choosing $\phi = \Nil_x \varphi_q$ gives us the desired isometry of $\Nil(x)$.
\end{proof}

By the same argument as in the above proof, we get that any $\varphi \in \Isom(M)$ induces an isometry $\Nil_x \varphi: \Nil(x) \to \Nil(\varphi(x))$. In particular, all of the spaces $\Nil(x)$, $x \in M$ are isometric. We will therefore write this space simply as $\Nil(M)$. We will call this space the tangent cone or nilpotentization of $M$. Observe that $\Nil(\lambda M) = \Nil(M)$.

\section{Model spaces with $\sfF_{n,r}$ as their tangent cone} \label{sec:Free}
\subsection{Loose model spaces}
In Section~\ref{sec:Frame} and Section~\ref{sec:RolSum} we will consider operations on model spaces that do not necessarily preserve the bracket-generating condition. We will therefore make the following definition.
\begin{definition}
Let $(M,D, g)$ be the sub-Riemannian manifold where $D$ is not necessarily bracket-generating.
\begin{enumerate}[\rm (i)]
\item We say that $\varphi: M \to M$ is a loose isometry if $\varphi$ is a diffeomorphism such that
$$\varphi_* D = D, \qquad \langle \varphi_* v, \varphi_* w \rangle_g = \langle v, w \rangle_g,$$
for any $v,w \in D$.
\item If $M$ has a smooth action of a Lie group $G$, we will say that~$G$ acts loosely isometric if for every $\varphi \in G$, the map $x \mapsto \varphi \cdot x$ is a loose isometry.
\item We say that $M$ is a loose homogeneous space if there is a Lie group~$G$ acting transitively and loosely isometric on $M$.
\item We say that $M$ is a loose model space if there is a Lie group $G$ acting loosely isometric on $M$ such that for every $x,y \in M$ and $q \in \Ort(D_x, D_y)$ we have $\varphi \cdot v =qv$ for some $\varphi \in G$ and any $v \in D_x$.
\end{enumerate}
\end{definition}

Even though loose homogeneous spaces are not metric spaces in general, we obtain true homogeneous spaces by restricting to each orbit.
\begin{lemma} \label{lemma:Loose}
Let $(M, D, g)$ be a loose homogeneous space.
For any $x \in M$, define
$$\calO_x = \{ y\in M \, : \,\text{there is a horizontal curve connecting $x$ and $y$} \}.$$
The following then holds.
\begin{enumerate}[\rm (a)]
\item The subbundle $D$ is equiregular.
\item For any $x \in M$, $(\calO_x, D|\calO_x, g|\calO_x)$ is a sub-Riemannian homogeneous space. Furthermore, all of these spaces are isometric.
\item If $(M, D, g)$ is a loose model space, then the universal cover $\tilde \calO_x$ of $\calO_x$ with the lifted sub-Riemannian structure is a model space for any $x \in M$. Furthermore, all of these model spaces are isometric.
\end{enumerate}
\end{lemma}
Note that if $M$ is simply connected in Lemma~\ref{lemma:Loose}~(c), this does not guarantee that the orbits are simply connected. For this reason, we have left out the requirement of simply connectedness in the definition of loose model spaces.

\begin{proof}
Let us first prove (a). Define $\underline{D}^j$ as in Section~\ref{sec:SRdef} and observe that $\Ad(\varphi)$ preserves $\underline{D}^j$ by Remark~\ref{re:Dj}. It follows that if $n_j(x) = \rank\{ Y(x) \, : \, Y \in \underline{D}^j \}$ then $n_j(x) = n_j(\varphi(x))$, so $D$ is equiregular. In particular, if $r$ is the step of $D$, then $D^r$ is integrable with foliations given by $\{ \calO_x \, : \, x \in M\}$. Hence, $D|\calO_x$ is bracket generating when considered as a subbundle of $T\calO_x$.

To prove (b) and (c), observe that for any horizontal curve $\gamma:[0,1] \to M$ with $\gamma(0) = x$, we have that $\varphi(\gamma)$ is a horizontal curve starting at $\varphi(x)$. It follows that $\varphi(\calO_x) = \calO_{\varphi(x)}$, and, in particular, $\varphi(\calO_x) = \calO_x$ whenever $\varphi(x) \in \calO_x$. Hence, any loose isometry taking a point in $\calO_x$ to another point in the same orbit gives us an isometry of the orbit. The result follows.
\end{proof}

\begin{remark}
The proof of Proposition~\ref{prop:Connection} does not depend on the bracket-generating condition, so loose model spaces also have canonical partial connections.
\end{remark}

\subsection{Model spaces induced by frame bundles} \label{sec:Frame}
Let $(M, D, g)$ be a model space with canonical partial connection $\nabla$. Consider the bundle of orthonormal frames $F^{\Ort}(D)$ of $D$ and write $G = \Isom(M)$ for the isometry group, acting freely and transitively on~$F^{\Ort}(D)$. 

Let $\calE$ be the subbundle of $TF^{\Ort}(D)$ corresponding to $\nabla$. Since $\pi_*|\calE_f : \calE_f \to D_x$ is an invertible linear map for any $f \in F^{\Ort}(D)_x$, we can lift $g$ to a metric $\tilde g$ defined on $\calE$. The sub-Riemannian structure $(\calE, \tilde g)$ is then invariant under the left action of $G$ and the right action of $\Ort(n)$. By combining these group actions, for any $q \in \Ort(\calE_{f_1},\calE_{ f_2})$, $f_1, f_2 \in F^{\Ort}(D)$, there is a loose isometry $\tilde \varphi$ such that $\tilde \varphi_* |\calE_{f_1} = q$. As a consequence, any connected component of $(F^{\Ort}(D), \calE, \tilde g)$ is a loose model space.
We define $\Frame(M)$ as the universal cover of one of these loose model spaces. Since there is a loose isometry between the connected components of $F^{\Ort}(D)$, $\Frame(M)$ is independent of choice of component. Also, $\Frame(\lambda M) = \lambda \Frame(M)$ for any $\lambda >0$ by definition.

\begin{proposition}
Let $(M, D, g)$ be a sub-Riemannain model space of step $r$ with canonical connection $\nabla$. Then $\Frame(M)$ is a (true) sub-Riemannian model space if and only if $\Hol^\nabla(x) = \SO(D_x)$ for some $x \in M$. In this case, $\Frame(M)$ has step $r$ or $r+1$. If $\Hol^\nabla(x) = \id_{D_x}$, then the orbits of the horizontal bundle of $\Frame(M)$ are isometric to~$M$.
\end{proposition}
Notice that if $\Frame(M)$ is a model space with canonical partial connection $\tilde \nabla$, then $\tilde \nabla$ has trivial holonomy. Hence, we can consider $M \mapsto \Frame(M)$ as a map sending model spaces of full holonomy to model spaces of trivial holonomy.

\begin{proof}
This result follows from the proof of Proposition~\ref{prop:HolCC}, stating that if $\Hol^\nabla(x) = \SO(D_x)$ then $\calE$ is bracket-generating and the only other option is that $\calE$ is integrable. Furthermore, write $G =\Isom(M)$ with Lie algebra $\mathfrak{g}$. Choose a frame $f_0 \in F^{\Ort}(D)_{x_0}$ to identify $G$ with $F^{\Ort}(D)$. Let $\pi: G \to M$ be the projection $\pi(\varphi) = \varphi(x_0)$. Write $K = G_{x_0}$ with Lie algebra $\mathfrak{k}$ and let $\mathfrak{p}^1 \subseteq \mathfrak{g}$ be the $K$-invariant subspace corresponding to $\nabla$. If we define $\mathfrak{p}
^{j+1} = \mathfrak{p}^j + [\mathfrak{p}^1, \mathfrak{p}^j]$ with $j \geq 1$, then $\mathfrak{p}^j \cap \mathfrak{k}$ must also be $K$ invariant as well. Hence, this intersection equals $0$ or $\mathfrak{k}$. In particular, if $\mathfrak{p}^1$ generates $\mathfrak{g}$ then $\mathfrak{g}$ equals $\mathfrak{p}^r$ or $\mathfrak{p}^{r+1}$ since $\pi_* \mathfrak{p}^r = T_{x_0} M$.
\end{proof}

\subsection{Model spaces with $\Nil(M) = \sfF_{n,r}$} Recall the definition of the free nilpotent Lie algebra $\sff_{n,r}$ and the corresponding sub-Riemannian model space $\sfF_{n,r}$ of Section~\ref{sec:Carnot1}.
We begin with the following important observation. Let $\mathfrak{p}^1$ be any $n$-dimensional subspace of a Lie algebra $\mathfrak{g}$. Define $\mathfrak{p}^{k+1} = \mathfrak{p}^k + [\mathfrak{p}^1, \mathfrak{p}^{k}]$, $k \geq 1$. By definition of the free algebra, we must have
$$\rank \mathfrak{p}^k \leq \rank \sff_{n,k},$$
for any $k \geq 1$. This relation has the following consequence. Let $M$ be a manifold and let $E$ be a subbundle of $TM$ of rank $n$. Let $\underline{n}(x) = (n_1(x), n_2(x), \dots )$ be the growth vector of $E$. Then $n_k(x) \leq \rank \sff_{n,k}$.

\begin{theorem} \label{th:Free}
Let $(M, D, g)$ be a sub-Riemannian model space with tangent cone $\Nil(x)$ isometric to $\sfF_{n,r}$ for every $x \in M$. Let $\nabla$ be the canonical partial connection on $D$. Write $G = \Isom(M)$ with Lie algebra $\mathfrak{g}$.
\begin{enumerate}[\rm (a)]
\item There is a unique affine connection $\bnabla$ on $D$, invariant under the action of $G$, satisfying $\bnabla_{|D} = \nabla$ and
\begin{equation} \label{RDr0} R^{\bnabla}(v,w) = 0, \qquad \text{for any $(v,w) \in D^i \oplus D^j$, $i +j \leq r$.}\end{equation}
\item If $\Hol^\nabla(x) = \SO(D_x)$ for some $x \in M$, then $\Frame(M)$ is a model space of step $r+1$.
\item Assume that $r$ is even. Then $\Hol^\nabla(x) = \id_{D_x}$ for every $x \in M$. As a consequence, for every $x \in M$, there is a Lie group structure on $M$ such that $(D,g)$ is left invariant, $x$ is the identity and every isometry fixing $x$ is a Lie group automorphism.
\end{enumerate}
\end{theorem}
\begin{proof}
Let $x_0$ be an arbitrary point in $M$ and $f_0 \in F^{\Ort}(M)_{x_0}$ a choice of orthonormal frame. Use this frame to identify $F^{\Ort}(M)$ with $G$. Write $\pi: G \to M$ for the projection $\varphi \mapsto \varphi(x_0)$. The partial connection $\nabla$ corresponds to a left invariant subbundle on $G$ obtained by left translation of a subspace $\mathfrak{p}^1$ such that $\pi_* \mathfrak{p}^1 = D_{x_0}$. Write $K = G_{x_0}$ with Lie algebra $\mathfrak{k}$.

\begin{enumerate}[\rm (a)]
\item Define $\mathfrak{p}^2, \dots, \mathfrak{p}^r$ by
$$\mathfrak{p}^{j+1} = \mathfrak{p}^j + [\mathfrak{p}^1, \mathfrak{p}^j].$$
By definition, we must have $\pi_* \mathfrak{p}^j = D_{x_0}^j$, so $\rank D_{x_0}^j \leq \rank \mathfrak{p}^j$. However, since the growth vector of $D$ equals that of the free nilpotent group, it follows that $\mathfrak{p}^j$ and $ D_{x_0}^j$ are of equal rank for $1 \leq j \leq r$. In particular, $\mathfrak{p} = \mathfrak{p}^r$ satisfies $\mathfrak{g} = \mathfrak{p} \oplus \mathfrak{k}$, so $\mathfrak{p}$ corresponds to an invariant affine connection. Furthermore, $[\mathfrak{p}^i, \mathfrak{p}^j ] \subseteq \mathfrak{p}$ whenever $i + j \leq r$ by definition. The result follows.
\item The statement follows simply from $\mathfrak{g} = \mathfrak{p}^r \oplus \mathfrak{k}$.
\item Define $\sigma \in G= \Isom(M)$ as the unique element satisfying $\sigma(x_0) = x_0$ and $\sigma_* |D_{x_0} = - \id_{D_{x_0}}$. Write $\mathfrak{p}^1$ and $\mathfrak{p}$ for the respective subspaces of $\mathfrak{g}$ corresponding to the partial connection $\nabla$ and the connection $\bnabla$ in~(a). Since $\sigma \in K$ and $\mathfrak{p}$ is $K$-invariant, we have $\Ad(\sigma) \mathfrak{p} \subseteq \mathfrak{p}$. Write $\mathfrak{p} = \mathfrak{p}^- \oplus \mathfrak{p}^+$ for the eigenspace decomposition of $\mathfrak{p}$. We then have
$$[\mathfrak{p}^+, \mathfrak{p}^+] \subseteq \mathfrak{p}^+ \oplus \mathfrak{k}, \qquad [\mathfrak{p}^-, \mathfrak{p}^-] \subseteq \mathfrak{p}^+ \oplus \mathfrak{k}.$$ 
$$[\mathfrak{p}^-, \mathfrak{p}^+] \subseteq \mathfrak{p}^-, \qquad [\mathfrak{p}^-, \mathfrak{k}] \subseteq \mathfrak{p}^-, \qquad [\mathfrak{p}^+, \mathfrak{k}] \subseteq \mathfrak{p}^+ .$$ 
Define $\mathfrak{a} = \mathfrak{p}^1 \subseteq \mathfrak{p}^-$ and for any $j=1,2 \dots, r-1$,
$$\mathfrak{a}^{j+1} = \left\{ [A,B] \, :\, A \in \mathfrak{a}^1, B \in \mathfrak{a}^j \right\}$$
Observe that $\mathfrak{a}^j$ is in $\mathfrak{p}^-$ (resp. $\mathfrak{p}^+$) whenever $j$ is odd (even). Furthermore, since the rank of $\mathfrak{p}^j$ grows as fast as the free group for $1 \leq j \leq r$, we have $\mathfrak{p}^{j+1} = \mathfrak{p}^j \oplus \mathfrak{a}^j$. Finally, it follows that $[\mathfrak{p}, \mathfrak{p}] \subseteq \mathfrak{p}$ if and only if $[\mathfrak{a}^1, \mathfrak{a}^r] \subseteq \mathfrak{p}$. However, $[\mathfrak{a}^1, \mathfrak{a}^r]$ is always contained in $\mathfrak{p}$ when $r$ is even, since this bracket must be contained in $\mathfrak{p}^-$. As a consequence, $R^{\bnabla} =0$ when $r$ is even, and the result follows from Proposition~\ref{prop:Lie}. \end{enumerate}
\end{proof}

\subsection{All model spaces in step~2}
We will now describe all model spaces $(M,D,g)$ for the case of step~$2$. By Example~\ref{ex:Step2} the tangent cone is isometric to $\sfF_{n,2}$ for each point $x \in M$.

\begin{theorem} \label{th:Step2}
Let $(M, D, g)$ be sub-Riemannian model space of step~$2$. Then $(M, D,g)$ is isometric to $\sfF_{n,2}$ or $\Frame(\sfSigma_n(\rho))$ for some $\rho \neq 0$.
\end{theorem}
In other words, model spaces of step~$2$ are $\sfF_{n,2}$ and the Lie groups $\sfG_n(\rho)$, $\rho \neq 0$ defined as in Section~\ref{sec:Riemannian}.

\begin{proof}
Write $G = \Isom(M)$. Using Theorem~\ref{th:Free}, we know that there is a Lie group structure on $M$ such that $(D,g)$ is left invariant. Write $1$ for the identity of $M$ and let $\mathfrak{m}$ be its Lie algebra. Let $\mathfrak{m}^1 \subseteq \mathfrak{m}$ be the inner product space such that $D_x = x \cdot \mathfrak{m}^1$. Recall from Theorem~\ref{th:Free} that every $\varphi \in G_1$ is a group automorphism.
Let $\sigma$ be the automorphism determined by $\sigma(1) =1$ and $\sigma_* |\mathfrak{m}^1 = -\id_{\mathfrak{m}^1}$. Let $\mathfrak{m} = \mathfrak{m}^- \oplus \mathfrak{m}^+$ be the corresponding decomposition into eigenspaces. It follows that $\mathfrak{m}^1 = \mathfrak{m}^-$, $\mathfrak{m}^+ = [\mathfrak{m}^1, \mathfrak{m}^1]$ and we have relations
$$[\mathfrak{m}^{+}, \mathfrak{m}^{\pm} ] = \mathfrak{m}^{\pm} , \qquad [\mathfrak{m}^-, \mathfrak{m}^{\pm} ] = \mathfrak{m}^{\mp}.$$

Furthermore, since $\Nil(1)$ is isometric to $\sfF_{n,2}$, we can identity the vector spaces $\mathfrak{m}^-$ and $\mathfrak{m}^+$ with respectively $\mathbb{R}^n$ with the usual euclidean structure and $\ort(n)$. If we write elements as $(x, A) \in \mathfrak{m}^- \oplus \mathfrak{m}^+$, $A \in \ort(n)$, $x \in \mathbb{R}^n$, we must have
$$[(x, 0), (y, 0) ] = (0,  x \wedge y),$$
by our characterization of the nilpotentization.
Note that for every $q \in \Ort(\mathfrak{m}^-)$, which we can identify with an element in $\Ort(n)$, there is a corresponding group isomorphism $\varphi_q$ having $1$ as fixed point. This map acts on $\mathfrak{m}$ by
$${\varphi_q}_* (x, A) = (qx, qAq^{-1}), \qquad x \in \mathbb{R}^n, A \in \ort(n),$$
and this map has to be a Lie algebra automorphism. By this fact, it follows from Lemma~\ref{lemma:tensors}, there are constants $\rho_1$ and $\rho_2$ such that
$$[(x, A), (y,B) ] = ( \rho_1 (A y - Bx), x \wedge y + \rho_2 [A, B] ).$$
Using the Jacobi identity, we must have $\rho_1 = \rho_2 = \rho$. The result follows by Section~\ref{sec:Riemannian}.
\end{proof}

\section{Model spaces with $\sfC_{n,r}$ as their tangent cone} \label{sec:Rol}
\subsection{All model spaces with $\sfC_{n,3}$ as tangent cone} \label{sec:Step3} In this section, we will study model spaces whose nilpotentization equal $\sfC_{n,r}$ of Example~\ref{ex:FirstC} and~\ref{ex:Rol}.
To begin with, we consider a sub-Riemannian model space $(M, D, g)$ with $\Nil(M)$ isometric to $\sfC_{n,3}$.
Write $G = \Isom(M)$ for the isometry group and let $K = G_{x_0}$ be the stabilizer group for some chosen $x_0 \in M$. Let $\mathfrak{g}$ and $\mathfrak{k}$ be the respective Lie algebras of $G$ and $K$. Write $\pi: G \to M$ for the map $\pi(\varphi) = \varphi(x_0)$.

Define $\sigma$ as unique isometry satisfying $\sigma(x_0) = x_0$ and $\sigma_*|_{D_{x_0}} = - \id_{D_{x_0}}$. Let $\mathfrak{g} =\mathfrak{g}^- \oplus \mathfrak{g}^+$ be the eigenvalue decomposition and notice that $\mathfrak{k} \subseteq \mathfrak{g}^+$. Let $\mathfrak{a}^1$ be the subspace of $\mathfrak{g}^- \subseteq \mathfrak{g}$ corresponding to the canonical partial connection $\nabla$. Let $\mathfrak{a}^3$ be a $K$-invariant complement of $\mathfrak{a}^1$ in $\mathfrak{g}^-$. Define
$$\mathfrak{a}^2 = [\mathfrak{a}^1, \mathfrak{a}^1] \subseteq \mathfrak{g}^+.$$
This has to be transverse to $\mathfrak{k}$, since $\pi_* (\mathfrak{a}^1 \oplus \mathfrak{a}^2) = D^2_{x_0}$ with both spaces having rank $\frac{n(n+1)}{2}$. As a result we have
\begin{equation} \label{Decomposition} \mathfrak{g} = \mathfrak{a}^1 \oplus \mathfrak{a}^2 \oplus \mathfrak{a}^3 \oplus \mathfrak{k}.\end{equation}

As vector spaces, we will identify $\mathfrak{a}^1$ and $\mathfrak{a}^3$ with~$\mathbb{R}^n$ and identify~$\mathfrak{a}^2$ and~$\mathfrak{k}$ with~$\ort(n)$. According to the decomposition \eqref{Decomposition}, we write an element in~$\mathfrak{g}$ as $(x, A, u, C)$ with $x,u \in \mathbb{R}^n$ and $A,C \in \ort(n)$. Since $\Nil(x_0)$ is isometric to $\sfC_{n,3}$, we have
$$[(x,0,0,0), (y, 0,0,0)] = (0, x \wedge y,0,0).$$
and
$$\pr_{\mathfrak{a}_3} [(0,A,0,0), (x, 0,0,0)] = (0, 0,Ax,0).$$
As a consequence, if $\varphi_q \in K$ is the unique element satisfying ${\varphi_q}_* |D_{x_0} = q \in \Ort(D_{x_0})$, then
$${\varphi_q}_* (x, A, u,C) = (qx, \Ad(q)A, qu, \Ad(q)C).$$
Using Lemma~\ref{lemma:tensors}, we know that there is some constant $c$ such that
$$[(0,A,0,0), (x, 0,0,0)] = (c Ax, 0,Ax,0).$$
By replacing $\mathfrak{a}^3$ with elements of the form $(cu, 0, u,0)$, we may assume~$c = 0$.

Using Lemma~\ref{lemma:tensors}, we know that there exists constant $a_j$, $b_j$, $c_j$ and $d_j$, for $j =1,2$ such that
\begin{align*}
\left[ \! \left(\begin{array}{c}  x \\ A \\ u \\0 \end{array} \right)^\transpose \! \! , \! \left( \begin{array}{c} y \\ B \\ v \\ 0 \end{array} \right)^\transpose  \right] & \! = \! \left( \begin{array}{c} c_1 (Av - Bu) \\
 x \wedge y + a_1(x \wedge v + u \wedge y) + b_1 [A,B] + d_1 u \wedge v \\
 Ay -Bx + c_2 (Av - Bu) \\
a_2(x \wedge v+ u \wedge y) + b_2[A,B] + d_2 u \wedge v \end{array} \right)^\transpose
\end{align*}

From the Jacobi identity, we have
$$a_1 = b_1 = c_2, \qquad a_2 = b_2 = c_1,$$
and
$$d_1 = a_2 + a_1^2 ,\qquad d_2 = a_1 a_2,$$
Hence, a choice of $a_1$ and $a_2$ uniquely determines all the constants. We will write the corresponding Lie algebra as $\mathfrak{g}(n,a_1, a_2)$ and the corresponding sub-Riemannian model space $M(n,a_1, a_2)$.

Next, we need to show that the parameters $a_1, a_2$ determine $M(n,a_1, a_2)$ up to isometry. Let $\Phi: M = M(n,a_1, a_2) \to \tilde M = M(n,\tilde a_1, \tilde a_2)$ be an isometry. By Remark~\ref{re:Compare} we then know that there is a Lie algebra isomorphism
$$\Psi:  \mathfrak{g} = \mathfrak{g}(n,a_1, a_2) = \mathfrak{a}^1 \oplus \mathfrak{a}^2 \oplus \mathfrak{a}^3 \oplus \mathfrak{k} \to \tilde {\mathfrak{g}} = \mathfrak{g}(n,\tilde a_1, \tilde a_2) = \tilde {\mathfrak{a}}^1 \oplus \tilde {\mathfrak{a}}^2 \oplus \tilde {\mathfrak{a}}^3 \oplus \tilde {\mathfrak{k}}, $$
that maps $\mathfrak{a}^1$ isometrically on $\tilde {\mathfrak{a}}^1$ and satisfies $\Psi(\mathfrak{k}) = \tilde {\mathfrak{k}}$. Since $\mathfrak{a}^2 = [\mathfrak{a}^1, \mathfrak{a}^1]$ and $\mathfrak{a}^3 = [\mathfrak{a}^1, \mathfrak{a}^2]$, with similar relations in $\tilde {\mathfrak{g}}$, it follows that if we identify $\mathfrak{a}^1$ and $\tilde{\mathfrak{a}}^1$ through $\Psi|\mathfrak{a}^1$, we get an identification of $\mathfrak{a}^j$ with $\tilde {\mathfrak{a}}^j$ as well for $j=2,3$. In conclusion, we must have $a_1= \tilde a_1$ and $a_2 = \tilde a_2$.

The theorem below summarizes the above discussion.
\begin{theorem} \label{th:Rol3}
If $M$ is a sub-Riemannian model space with $\Nil(M)$ isometric to $\sfC_{n,3}$, then $M$ is isometric to the space $M(n,a_1, a_2)$ for some $(a_1, a_2) \in \mathbb{R}^2$.
\end{theorem}
We remark the relation $\lambda M(n, a_1, a_2) = M( n, a_1/\lambda^2, a_2/\lambda^4)$ for any $\lambda >0$. Also notice that if $\nabla$ is the canonical partial connection of $M(n,a_1,a_2)$, then $\Hol^\nabla(x)$ is trivial for an arbitrary $x \in M(n, a_1, a_2)$ if and only if $a_2 = 0$.

\subsection{Rolling sums of model spaces} \label{sec:RolSum}
For $j=1,2$, let $(M^{(j)}, D^{(j)}, g^{(j)})$ be a loose model spaces with canonical partial connection $\nabla^{(j)}$ on $D^{(j)}$. If both $D^{(1)}$ and $D^{(2)}$ have the same rank $n$, we introduce a new loose model space $M_1 \boxplus M_2$ whose horizontal bundle also have rank $n$.

Introduce a manifold
\begin{align*}
Q = \Ort(D^{(1)}, D^{(2)}) & = \left\{ q: D_x^{(1)} \to D_y^{(2)} \, \, : \, \begin{array}{c} (x,y) \in M^{(1)} \times M^{(2)} \\ \text{$q$ linear isometry} \end{array} \right\} .
\end{align*}
On this space, we define a horizontal bundle $D$ in the following way. The smooth curve $q(t): D_{\gamma^{(1)}(t)}^{(1)} \to D_{\gamma^{(2)}(t)}^{(2)}$ is tangent to $D$ if and only if
\begin{enumerate}[\rm (i)]
\item $\gamma^{(j)}$ is tangent to $D^{(j)}$ for $j=1,2$,
\item $q(t) \dot \gamma^{(1)}(t) = \dot \gamma^{(2)}(t)$,
\item for every vector field $V(t)$ with values in $D^{(1)}$ along $\gamma^{(1)}$,
$$\nabla_{\dot \gamma^{(2)}(t)}^{(2)} q(t) V(t) = q(t) \nabla_{\dot \gamma^{(1)}(t)}^{(1)} V(t).$$
\end{enumerate}
From the first order ordinary differential equations in (ii) and (iii), it follows that~$q(t)$, including its image $\gamma^{(2)}$ in $M^{(2)}$, is uniquely determined by $\gamma^{(1)}$ up to initial condition. Hence, the linear map
$${\pi^{(1)}}_*|D_q: D_q \to D^{(1)}_{\pi^{(1)}(q)},$$
is invertible since curves tangent to $D^{(1)}$ have well defined horizontal lifts tangent to $D$. We can use this fact to pull back the metric $g^{(1)}$ from $D^{(1)}$ to a metric~$g$ on~$D$. This equals the result of pulling the metric back from~$D^{(2)}$.

Equivalently, $Q$ can be realized as $\big(F^{\Ort}(D^{(1)}) \times F^{\Ort}(D^{(2)})\big)/\Ort(n)$ where the quotient is with respect to the diagonal action. Consider the subbundle $\calE^{(j)}$ of $TF^{\Ort}(D^{(j)})$ corresponding to~$\nabla^{(j)}$. Let $\pi^{(j)}:  F^{\Ort}(D^{(j)}) \to M^{(j)}$ denote the projection. Write $e_1, \dots, e_n$ for the standard basis of $\mathbb{R}^n$. The subbundle $\calE^{(j)}$ has a global basis of elements $X_1^{(j)}, \dots, X_n^{(j)}$ such that such that
\begin{equation} \label{FrameBasis} {\pi^{(j)}}_* X^{(j)}_i(f) = f(e_i), \qquad f \in F^{\Ort}(D^{(j)}).\end{equation}
If we pull back the metric $g^{(j)}$ to $\calE^{(j)}$, then $X_1, \dots, X_n$ forms an orthonromal basis. We can then define $(D,g)$ such that $\{ \nu_*  (X_i^{(1)} +  X_i^{(2)}) \, : \, 1 \leq i \leq n \}$ forms an orthonormal basis at each point with $\nu: F(D^{(1)}) \times F(D^{(2)}) \to Q$ being the quotient map.

By combining isometries of $M^{(1)}$ and $M^{(2)}$, we obtain that any connected component of $(Q, D, g)$ is a loose model space. More precisely, for any pair of isometries $\varphi^{(j)} \in \Isom(M^{(j)})$, $j=1,2$, the map
$$\Phi_{\varphi^{(1)}, \varphi^{(2)}}: q \in \Ort(D_x^{(1)}, D_y^{(2)}) \mapsto {\varphi^{(2)}}_*(y) \circ q \circ {\varphi^{(1)}}_*(x),$$
is an isometry of $(Q,D,g)$. It is simple to verify that any $\tilde q \in \Ort(D_{q_1}, D_{q_2})$ can be represented by an isometry of this type. We write $M^{(1)} \boxplus M^{(2)}$ for the loose model space given as the universal cover of one of the connected components of $(Q, D, g)$. We will call the resulting space \emph{the rolling sum} of $M^{(1)}$ and $M^{(2)}$. In the special case when $M^{(1)}$ and $M^{(2)}$ are Riemannian model spaces, $M^{(1)} \boxplus M^{(2)}$ correspond to the optimal control problem of rolling $M^{(1)}$ on $M^{(2)}$ without twisting or slipping along a curve of minimal length. For more information, see e.g. \cite{Gro16,CiKo12,JuZi08}.

Observe the following relations, where $\cong$ denotes loose isometry.
\begin{enumerate}[\rm (i)]
\item (Commutativity) $M^{(1)} \boxplus M^{(2)} \cong M^{(2)} \boxplus M^{(1)}$,
\item (Associativity) $(M^{(1)} \boxplus M^{(2)}) \boxplus M^{(3)} \cong M^{(1)} \boxplus (M^{(2)} \boxplus M^{(3)})$
\item (Distributivity) $\lambda (M^{(1)} \boxplus M^{(2)}) \cong \lambda M^{(1)} \boxplus \lambda M^{(2)}$ for any $\lambda >0$.
\end{enumerate}
Commutativity and distributivity follows from the definition. As for associativity, consider sub-Riemannian loose model spaces $(M^{(j)}, D^{(j)}, g^{(j)})$, $j=1,2,3$. Let $\calE^{(j)}$ be the subbundle of $TF^{\Ort}(D^{(j)})$ corresponding to the canonical partial connection on $M^{(j)}$ and let $X^{(j)}_1, \dots, X^{(j)}_n$ be its basis defined by \eqref{FrameBasis}. Define $(Q,D,g)$ as above. Then $F^{\Ort}(D) = F^{\Ort}(D^{(1)}) \times F^{\Ort}(D^{(2)})$. It follows that $(M^{(1)} \boxplus M^{(2)}) \boxplus M^{(3)}$ is the universal cover of a connected component of
$$\Big(F^{\Ort}(D^{(1)}) \times F^{\Ort}(D^{(2)}) \times F^{\Ort}(D^{(3)}) \Big)/\Ort(n),$$
with sub-Riemannian structure given as the projection of the structure on $F^{\Ort}(D^{(1)}) \times F^{\Ort}(D^{(2)}) \times F^{\Ort}(D^{(3)})$ with $X^{(1)}_i + X^{(2)}_i + X^{(3)}_i$, $i =1, \dots, n$, as an orthonormal basis. Associativity follows.

\subsection{Rolling sum of Riemannian model spaces} \label{sec:NilRol}
We will end this section with the relation between rolling sums of Riemannian model spaces and model spaces with nilpotentization $\sfC_{n,r}$.

\begin{theorem}
Let $\rho_1, \dots, \rho_r$ be real numbers and define $r \times r$-matrices,
$$\boldsymbol \rho = (\rho_i^{j-1}),\qquad \boldsymbol \mu = (\rho_i^j), \qquad i,j = 1, \dots, r,$$
with the convention that $\rho_i^0 =1$ for any value of $\rho_i$.
\begin{enumerate}[\rm (a)]
\item We have $\det \boldsymbol \rho \neq 0$ if and only if
$$M = \sfSigma_n(\rho_1) \boxplus \cdots \boxplus \sfSigma_n(\rho_r)$$
is a (true) sub-Riemannian model space. Furthermore, if $\det \boldsymbol \rho \neq 0$, then the tangent cone of $M$ at any point is isometric to $\sfC_{n,2r-1}$. \\
The loose model space
$$\tilde M = \Frame(\sfSigma_n(\rho_1) \boxplus \cdots \boxplus \sfSigma_n(\rho_r))$$ is a sub-Riemannian model space with tangent cone $\sfC_{n,2r}$ if and only if $\det \boldsymbol \rho \neq 0$ and $\det \boldsymbol \mu \neq 0$.
\item Let $\rho_1$ and $\rho_2$ be two real numbers such that $\rho_1 - \rho_2 \neq 0$. Define $M(n, a_1, a_2)$ as in Section~\ref{sec:Step3}. Then $\sfSigma_n( \rho_1) \boxplus \sfSigma(\rho_2)$ is isometric to
$$M(n,\rho_1 + \rho_2,- \rho_1 \rho_2).$$
\end{enumerate}
\end{theorem}
In particular, $M(n,a_1, a_2)$ is isometric to $\sfSigma(\rho_1) \boxplus \sfSigma(\rho_2)$ for some $\rho_1$ and $\rho_2$ if and only if
$$a_1^2 + 4a_2 > 0.$$
Furthermore, $\sfSigma_n(\rho_1) \boxplus \sfSigma_n(\rho_2)$ is isometric to $\sfSigma(\tilde \rho_1) \boxplus \sfSigma( \tilde \rho_2)$ if and only if $(\tilde \rho_1, \tilde \rho_2)$ equals $(\rho_1, \rho_2)$ or $(\rho_2, \rho_1)$.
\begin{proof}
\begin{enumerate}[\rm (a)]
\item We first write $\sfSigma_n(\rho_j)$ as a symmetric space. For any $j =1, \dots, r$, define $\mathfrak{g}^{(j)}$ as the vector space $\mathbb{R}^n \oplus \ort(n)$, with Lie brackets
$$[(x, A), (y, B)] = (Ay- Bx , [A,B] + \rho_j x \wedge y).$$
Let $G^{(j)}$ be the corresponding connected, simply connected Lie group and let $K^{(j)}$ be the subgroup with Lie algebra $\{ (x, A) \in \mathfrak{g}^{(j)} \, : \, x = 0\}$. Define $G = G^{(1)} \times \cdots \times G^{(r)}$ and let $\mathfrak{g} = \mathfrak{g}^{(1)} \oplus \cdots \oplus \mathfrak{g}^{(r)}$ be its Lie algebra.

For elements $\underline{u} = (u_i), \underline{v} = (v_i) \in \mathbb{R}^r$, define $\underline{u} \odot \underline{v} \in \mathbb{R}^r$ by coordinate-wise multiplication $\underline{u} \odot \underline{v} := (u_i v_i)$. Define elements $\underline{1} = (1,\dots, 1)$ and $\underline{\rho} = (\rho_1, \dots, \rho_r)$ in $\mathbb{R}^n$. For any $x \in \mathbb{R}^n$ and $\underline{u} \in \mathbb{R}^r$, define $x(\underline{u}) \in \mathfrak{g}$ such that its component in $\mathfrak{g}^{(j)}$ equals $(u_i x, 0)$. For any $A \in \ort(n)$, define $A(\underline{u})$ similarly.
We then have bracket-relations in $\mathfrak{g}$ given by,
\begin{equation}
\label{BracketRolSum}
\left\{
\begin{array}{rcl}
{[x(\underline{u}), y(\underline{v})]} & =&  (x\wedge y)(\underline{\rho} \odot \underline{u} \odot \underline{v}), \\
{[A(\underline{u}), y(\underline{v})]} & =&  (Ay)(\underline{u} \odot \underline{v}), \\
{[A(\underline{u}), B(\underline{v})]} & =&  [A,B](\underline{u} \odot \underline{v}).
\end{array} \right.
\end{equation}

Consider the subspace $\mathfrak{p}^1= \{ x(\underline{1}) \, : \, x \in \mathbb{R}^n \}$ with inner product $\langle x(\underline{1}), y(\underline{1}) \rangle = \langle x, y\rangle$. Define a sub-Riemannian structure $(\tilde D, \tilde g)$ by left translation of $\mathfrak{p}^1$ and its inner product. Define the subalgebra $\mathfrak{k} = \{ A(\underline{1}) \, : \, A \in \ort(n)\}$ with corresponding subgroup $K$. Since the sub-Riemannian structure $(\tilde D, \tilde g)$ is $K$-invariant, we get a well defined induced sub-Riemannian structure $(D,g)$ on $G/K$. By choosing a reference orthonormal frame on each of the manifolds $\sfSigma_n(\rho_j)$, we may identify $\Frame(\sfSigma_n( \rho_1) \boxplus \cdots \boxplus \sfSigma_n(\rho_r))$ and $\sfSigma_n(\rho_1) \boxplus \cdots \boxplus \sfSigma_n(\rho_r)$ with respectively $(G, \tilde D, \tilde g)$ and $(G/K, D, g)$.

Write the quotient map as $\nu : G \to G/K$. Let $\hat{ \mathfrak{p}}$ denote the Lie algebra generated by $\mathfrak{p}^1$. Write $\underline{\rho}^{\odot j}$ for the $j$-th iterated $\odot$-product of $\underline{\rho}$ with itself with convention $\underline{\rho}^{\odot 0} = \underline{1}$. By \eqref{BracketRolSum}, we know that
$$\hat{\mathfrak{p}} = \spn \{ x( \underline{\rho}^{\odot j}) , A(\underline{\rho}^{\odot j+1}) \, : \, x \in \mathbb{R}^n, A \in \ort(n), j=0,\dots, r-1   \}.$$
Hence, $D$ is bracket generating if and only if $\nu_*$ maps $\hat {\mathfrak{p}}$ surjectively on the tangent space at $\nu(1)$. This happens if and only if the vectors $\underline{1}, \underline{\rho}, \dots, \underline{\rho}^{\odot r-1}$ are linearly independent which equals the condition $\det \boldsymbol \rho \neq 0$. Similarly, $\hat{\mathfrak{p}}$ equals $\mathfrak{g}$ if and only if $\det \boldsymbol \rho \neq 0$ and $\det \boldsymbol \mu \neq 0$.

For the final statement regarding the nilpotentization, define $x^{(i)} = x(\underline{\rho}^{\odot i-1})$ and $A^{(i)} = A(\underline{\rho}^{\odot i})$. Using the relations \eqref{BracketRolSum}, we see that the brackets of the elements  $x^{(i)}$ and $A^{(i)}$ satisfy \eqref{RolNil}, Example~\ref{ex:Rol}, with the only difference being that elements $x^{(i)}$ and $A^{(i)}$ in Example~\ref{ex:Rol} are defined to eventually be zero. The nilpotentizations of $\tilde D$ and $D$ follows.
 
\item From the results of (a), it follows that $\sfSigma_n(\rho_1) \boxplus \sfSigma_n(\rho_2)$ is a sub-Riemannian model space if and only if $\rho_1- \rho_2 \neq 0$.

We use the notation of Section~\ref{sec:Step3} and the proof of (a). Let $\mathfrak{g}$ be the Lie algebra of the isometry group of $\sfSigma_n( \rho_1) \boxplus \sfSigma_n(\rho_2)$ and let $\Psi: \mathfrak{g}(n,a_1, a_2) \to \mathfrak{g}$ be a Lie algebra isomorphism preserving the horizontal subspaces. Any such isometry, up to a coordinate change, has to be of the form
$$\Psi: (x, A, z, C) \in \mathfrak{g}(n,a_1, a_2) \mapsto x(\underline{1}) + A(\underline{\rho}) + z(\underline{\rho}) + C(\underline{1}).$$
We can then determine that $a_1 = \rho_1 + \rho_2$ and $a_2 = - \rho_1 \rho_2$ from
$$[x(\underline{1}), z(\underline{\rho})] = (x \wedge z)(\underline{\rho}^{\odot 2}) = (\rho_1 + \rho_2) (x \wedge z)(\underline{\rho}) - \rho_1 \rho_2 (x\wedge z)(\underline{1}).$$

\end{enumerate}
\end{proof}

\appendix
\section{Technical results related to $\Ort(n)$} \label{sec:On}
\subsection{Representations}
We will need the following result of representations of $\Ort(n)$. For details, see e.g. \cite[Chapter~2]{Don11}.
\begin{lemma} \label{lemma:Representation}
\begin{enumerate}[\rm (a)]
\item For any $n\geq 2$, define a representation $\sigma_n$ of $\Ort(n)$ on $\mathbb{R}^n$ by
$$\sigma_n(q) v = qv, \qquad q \in \Ort(n), v \in \mathbb{R}^n.$$
Then $\sigma_n$ is irreducible.
\item For any $n \geq 2$, define a representation $\psi_n$ of $\Ort(n)$ on $\ort(n)$ by
$$\psi_n(q) A = \Ad(q) A, \qquad q \in \Ort(n), A \in \ort(n).$$
Then $\psi_n$ is irreducible.
\item $\sigma_n$ is never isomorphic to $\psi_n$ for any $n \geq 2$.
\end{enumerate}
\end{lemma}
Notice that we can identify $\psi_n$ with $\wedge^2 \sigma_n$ through \eqref{wedge}.

\begin{remark} \label{re:Hodge}
We emphasize that these results are for representations of $\Ort(n)$. If we instead consider representations of $\SO(n)$, the above claims need to be modified in the following way for the cases $n =2$, $3$ and $4$. Consider $e_1, \dots, e_n$ as the standard basis of~$\mathbb{R}^n$. Introduce the Hodge star map $\star: \wedge^k \mathbb{R}^n \to \wedge^{n-k} \mathbb{R}^n$ defined such that
$$\alpha \wedge (\star \beta) = \langle \alpha , \beta \rangle e_1 \wedge \cdots \wedge e_n.$$
Note that $\star \star = (-1)^{k (n-k)} \id $.
\begin{enumerate}[\rm (i)]
\item If $\tilde \psi_n$ is the adjoint representation of $\SO(n)$, then it is irreducible if and only if $n \neq 2,4$. Obviously $\tilde \psi_2$ is trivial. For $n = 4$, by identifying $\mathfrak{o}(4)$ with $\wedge^2 \mathbb{R}^4$, we can consider $\star$ as an endomorphism of $\mathfrak{o}(4)$. Then $\ort(4) = \mathfrak{g}^- \oplus \mathfrak{g}^+$ has subrepresentations
$$\mathfrak{g}^{\pm} = \{ A \pm \star A \, : \, A \in \ort(4) \}.$$
\item If we define the representation $\tilde \sigma_n$ of $\SO(n)$ on $\mathbb{R}^n$ by $\tilde \sigma_n(q) x = qx$, then $\tilde \psi_n$ is not isomorphic to $\tilde \sigma_n$ for $n \neq 3$. For the case $n = 3$, however, we have the relation $\tilde \psi_3 = \star \tilde \sigma_3 \star$. Written in another way, if $\times$ denotes the cross-product on $\mathbb{R}^3$, then $\star (x \wedge y) = x \times y$ and hence
$$\tilde \psi_3(q)(x \wedge y) = \star \tilde \sigma_3(q)(x \times y).$$
\end{enumerate}
\end{remark}

\subsection{Invariant maps related to $\ort(n)$} 
We will need the following facts about invariant tensors related to $\ort(n)$. We note that we have chosen direct approach, as this avoids recalling further results of representation theory and dealing with the differences between representations of $\ort(n)$ and $\Ort(n)$.
\begin{lemma} \label{lemma:tensors}
For any $n \geq 2$, consider the Lie algebra $\ort(n)$.
\begin{enumerate}[\rm (a)]
\item Let $S: \ort(n) \to \ort(n)$ be any linear map such that
\begin{equation} \label{SCommute} S(\Ad(q)A) = \Ad(q) S(A), \end{equation}
for any $q \in \Ort(n)$ and $A \in \ort(n)$. Then $S = c \id$ for some constant $c \in \mathbb{R}$.
\item Let $T: \ort(n) \otimes \mathbb{R}^n \to \mathbb{R}^n$ be any linear map such that
\begin{equation} \label{TCommute} q T(A\otimes x) = T(\Ad(q) A \otimes qx) \nonumber ,\end{equation}
for any $A \in \ort(n)$, $x \in \mathbb{R}^n$ and $q \in \Ort(n)$. Then $T(A \otimes x) = c Ax$ for some constant $c \in \mathbb{R}$.
\item Let  $W: \wedge^2 \ort(n) \to \ort(n)$ be any linear map such that
\begin{equation} \label{PreDer} \Ad(q) W(A, B) = W( \Ad(q) A, \Ad(q)B) , \nonumber \end{equation}
for any $q \in \Ort(n)$ and $A, B \in \ort(n)$. Then $W(A,B) = c[A, B]$ for some $c \in \mathbb{R}$.
\end{enumerate}
\end{lemma}
\begin{proof}
\begin{enumerate}[\rm (a)]
\item Follows since the representations are irreducible.
\item Define $S: \ort(n) \to \gl(n)$ such that $S(A) = B = (B_{rs})$ if
$$B_{rs} = \langle e_r , T(A \otimes e_s) \rangle $$
Since $S$ then satisfies \eqref{SCommute}, all we need to show is that the image of $S$ is in $\ort(n)$. Note that \eqref{SCommute} implies $[B, S(A)] = S([B,A])$.

Write $S = S^+ + S^-$ such that $S^{\pm} = \frac{1}{2} (S(A) \pm S(A)^\transpose)$. Both $S^+$ and $S^-$ satisfy \eqref{SCommute}. By (a), we have $S^-(A) = c A$ for some constant $c$ so to complete the proof, we need to show that $S^+ = 0$. For any $x, y \in \mathbb{R}^n$, write $xy = y x^\transpose + x y^\transpose$. For any symmetric matrix $\mu$, we have
$$[x \wedge y, \mu] = (\mu x) y - x (\mu y).$$
Let $e_1, \dots, e_n$ be the standard basis and define $q_i \in \Ort(n)$ by,
$$q_i e_j = \left\{ \begin{array}{ll} - e_j & \text{if $i=j$} \\ e_j & \text{if $i \neq j$}\end{array} \right.$$
We note that
$$S^+(\Ad(q_i) e_i \wedge e_j) = - S^+(e_i \wedge e_j) = \Ad(q_i) S^+(e_i \wedge e_j).$$
It follows that $S^+(e_i \wedge e_j) \in \spn \{ e_i e_k \, : \, k \neq i \}$. By applying $q_j$, it follows that $S^+(e_i \wedge e_j ) \in \spn\{ e_j e_k \, : \, k \neq j \}$ as well, meaning that there are constants $\mu_{ij}$ with $\mu_{ij} = - \mu_{ji}$ such that
$$S^+(e_i \wedge e_j) =\mu_{ij} e_i e_j.$$
However, by the identity
$$0 = S^+([e_i \wedge e_j, e_i \wedge e_j ]) = [e_i \wedge e_j , \mu_{ij} e_i e_j] = \mu_{ij} (e_i e_i - e_j e_j),$$
it follows that $S^+ = 0$.

\item We will do the proof by induction. The statement is obviously true for $n=2$ since we must have $W =0$. Assume that the statement holds true on $\ort(n)$. Consider $\ort(n+1)$ as the Lie algebra of pairs $(x, A) \in \mathbb{R}^n \oplus \ort(n)$, with bracket relations
$$\Big[(x,A), (y, B) \Big] = (Ay - Bx, [A,B] + x \wedge y).$$
Any pair $(x,A)$ represents a matrix
$$\left( \begin{array}{cc} A & x \\ -x^\transpose & 0 \end{array} \right).$$
Consider a map $W: \wedge^2 \ort(n+1) \to \ort(n+1)$ that commutes with the adjoint action of $\Ort(n+1)$. Write $ W = (W_1, W_2)$ according to the splitting $\ort(n+1) = \mathbb{R}^n \oplus \ort(n)$.
Let $q \in \Ort(n)$ be any element and define elements $q^\pm \in \Ort(n+1)$ by
$$q^\pm = \left( \begin{array}{cc} q & 0 \\ 0 & \pm 1 \end{array} \right).$$
Then
$$\Ad(q^\pm)(x,A) = (\pm q x, \Ad(q) A).$$
Observe that
\begin{align*}
& W\Big( \Ad(q^\pm) (0,A), \Ad(q^\pm) (0, B)\Big)  = W\Big( (0,\Ad(q)A),  (0, \Ad(q) B)\Big) \\
& = \Big( \pm q W_1\big( (0, A), (0, B) \big) \, , \, \Ad(q) W_2\big((0, A), (0,B) \big) \Big).
\end{align*}
We hence have $W_2((0, A), (0,B)) = c_1[A,B]$ for some constant $c_1$ by our induction hypothesis. Furthermore, we must have $W_1((0, A), (0,B)) = 0$ since both sides must be independent of the sign in $q^\pm$.

From the fact that both sides must either depend or be independent of the sign in $q^\pm$, we also obtain that $W_1((x,0), (y,0)) = 0$ and $W_2((x,0), (0, A)) = 0$. The identification of the representation of $\Ort(n)$ on $\wedge^2 \mathbb{R}^n$ with the adjoint representation on $\ort(n)$ makes it clear that we must have $W_2((x,0), (y,0)) = c_2 x \wedge y$ for some constant $c_2$. Finally, there is a constant $c_3$ such that $W_1((x,0), (0,A)) = - c_3 Ax$ by (b).

Since permutations of coordinates are elements in $\Ort(n+1)$, we can use these maps to show that $c_1$, $c_2$ and $c_3$ coincide. The result follows.
\end{enumerate}
\end{proof}

\input

\bibliographystyle{habbrv}
\bibliography{Bibliography}

\end{document}